\documentclass[a4paper,11pt]{amsart}
\usepackage{amssymb}
\usepackage{latexsym}
\usepackage{amsmath}
\usepackage{enumerate}
\usepackage{amsmath, hyperref}

\usepackage{tikz}
\usepackage{geometry}
\usepackage[all,pdf]{xy}

\newtheorem{theorem}{Theorem}[section]
\newtheorem{lemma}[theorem]{Lemma}
\newtheorem{corollary}[theorem]{Corollary}
\newtheorem{proposition}[theorem]{Proposition}

\theoremstyle{definition}
\newtheorem{definition}[theorem]{Definition}
\newtheorem{question}[theorem]{Question}
\newtheorem{fact}[theorem]{Fact}

\newtheorem*{claim}{Claim}

\newtheorem*{convention}{Convention}

\begin{document}

\title{On Rosser theories}

\begin{abstract}
Rosser theories play an important role in the study of the incompleteness phenomenon and meta-mathematics of arithmetic. In this paper, we first define the notions of  $n$-Rosser theories, exact $n$-Rosser theories, effectively $n$-Rosser theories and effectively exact $n$-Rosser theories  (see Definition \ref{def of Rosser}). Our definitions are not restricted to arithmetic languages. Then we systematically examine properties of $n$-Rosser theories and relationships among them. Especially, we generalize some important theorems about Rosser theories for recursively enumerable sets in the literature to $n$-Rosser theories  in a general setting.

\end{abstract}



\author{Yong Cheng\\
School of Philosophy, Wuhan University, China}
\address{School of Philosophy, Wuhan University, Wuhan 430072 Hubei, Peoples Republic of China}

\email{world-cyr@hotmail.com}

\date{}

\thanks{I would like to thank Albert  Visser for helpful comments on the work. I would like to thank all referees for helpful suggestions and comments for improvements.}

\maketitle

\section{Introduction}

The notion of Rosser theories is introduced in \cite{Smullyan59} and \cite{Smullyan93}. Rosser theories play an important role in the study of the incompleteness phenomenon and meta-mathematics of arithmetic, and have important meta-mathematical properties (see \cite{Smullyan93}). All definitions of Rosser theories we know in  the literature are restricted to arithmetic languages which admit numerals for natural numbers. Even if Smullyan introduced the notion of Rosser theories  for $n$-ary relations in \cite{Smullyan93}, results about Rosser theories  in \cite{Smullyan93} are confined to 1-ary and two-ary relations. A general theory of Rosser theories  for $n$-ary relations for any $n\geq 1$ and relationships among them is missing in the literature.

In this paper, we first introduce the notion of $n$-Rosser theories in a general setting which generalizes the notion of Rosser theories for recursively enumerable (RE) sets. Then, we introduce the notions of exact $n$-Rosser theories, effectively $n$-Rosser theories and effectively exact $n$-Rosser theories (see Definition \ref{def of Rosser}). The notion of effectively $n$-Rosser theories (effectively exact $n$-Rosser theories) is an effective version of the notion of $n$-Rosser theories (exact $n$-Rosser theories). We define that $T$ is Rosser if $T$ is $n$-Rosser for any $n\geq 1$. Our definitions of these notions are not restricted to arithmetic languages admitting numerals for natural numbers. Then we systematically examine properties of $n$-Rosser theories and relationships among them. Especially, we generalize some important theorems about Rosser theories for RE sets in \cite{Smullyan93} to $n$-Rosser theories. For these generalizations, we need some  tools such as Theorem \ref{semi-DU is DU} and a generalized version of the Strong Double Recursion Theorem as in Theorem \ref{SDRT}.

In this paper, the generalizations of the notions about Rosser theories for RE sets in \cite{Smullyan93} consist of two aspects: (1) going from pairs of RE 
sets  to pairs of $n$-ary RE relations for any 
$n \geq 2$, and (2) going from theories in
the usual arithmetic language to theories $T$ in which we can interpret a very basic theory of numerals which allows
us to introduce numerals in $T$.

Our motivation of these generalizations is two-folds. Firstly, just like many theorems in recursion theory have been generalized from RE sets to $n$-ary RE relations, and from versions without parameters to versions with parameters, it is natural for us to consider the generalizations of results about  RE sets to results about $n$-ary RE relations for any $n\geq 1$. 
Secondly, definitions of Rosser theories in the literature are  restricted to arithmetic languages which admit numerals for natural numbers. Under this restriction, we do not even know whether $\mathbf{ZFC}$ is Rosser since the language of $\mathbf{ZFC}$ does not  admit numerals. Thus, it is natural for us to extend the notions of Rosser theories in an arithmetic language to the notions of Rosser theories in a non-arithmetic language which can interpret a very basic theory of numerals. 
In fact, one research methodology in \cite{Smullyan93} is studying the meta-mathematical properties of formal theories related to incompleteness (or undecidability) by proposing stronger or more general meta-mathematical properties: from essential undecidability, recursive undecidability, effective inseparability to the Rosser property (for the definitions of the relevant notions, we refer to \cite{Smullyan93}). 

The structure of this paper is as follows. In Section \ref{old def}, we present the definition of Rosser theories and list the main theorems about Rosser theories in the literature. In Section \ref{new def}, we give the new definitions of $n$-Rosser theories, exact $n$-Rosser theories, effectively $n$-Rosser theories, effectively exact $n$-Rosser theories and Rosser theories. In Section \ref{prelim}, we list definitions and facts we will use. In Section \ref{basic facts}, we prove some basic facts about Rosser theories under our new definitions.  
In Section \ref{SRDT}, we prove a generalized version of the Strong Double Recursion Theorem which is a main tool in Sections \ref{generalization of theorem on rosser} and \ref{exact Rosser}.
In Section \ref{generalization of theorem on rosser}, we prove some theorems about relationships among $n$-Rosser theories, exact $n$-Rosser theories, effectively $n$-Rosser theories and effectively exact $n$-Rosser theories. Especially, we generalize  Theorem \ref{effect Rosser}  to $n$-Rosser theories.  
In Section \ref{exact Rosser}, we generalize Putnam-Smullyan Theorem \ref{PST} to  $n$-Rosser theories and prove Theorem \ref{improve PST}. As a main tool of the proof of Theorem \ref{improve PST}, we first prove in Section \ref{semi-DU} that semi-$\sf DU$ implies $\sf DU$ for a disjoint pair of $n$-ary relations.  In Section \ref{apply of DU}, we examine applications of the result that semi-$\sf DU$ implies $\sf DU$ in meta-mathematics of arithmetic. We first prove Theorem \ref{improve PST} and then Theorem \ref{key improve thm} which essentially improves Theorem \ref{improve PST}. In Section \ref{equi under def}, we examine relationships among $n$-Rosser theories under the assumption that the pairing function  is strongly definable in the base theory.

\subsection{Definitions of Rosser theories in the literature}\label{old def}

In this paper, we  fix a way of G\"{o}del coding as developed in standard textbooks such as  \cite{Enderton 2001, Smullyan93}. Under this coding, any formula or expression has a unique code. For any formula $\phi$, we use $\ulcorner \phi\urcorner$ to denote the code of $\phi$.

In this section,  we assume that $T$ is a consistent RE theory in a language of arithmetic which admits numerals $\overline{n}$ for any $n\in\omega$.
In \cite{Smullyan93}, Smullyan introduced the notions of Rosser theories for RE sets and for $n$-ary relations ($n\geq 2$).

\begin{definition}[\cite{Smullyan93}]~\label{standard def of Rosser}
\begin{enumerate}[(1)]
\item We say $T$ is a \emph{Rosser theory  for RE sets}  if for any disjoint pair $(A,B)$ of RE sets, there exists a formula $\phi(x)$ with exactly one free variable such that if $n\in A$, then $T\vdash \phi(\overline{n})$, and if $n\in B$, then $T\vdash \neg\phi(\overline{n})$.
  \item We say $T$ is a \emph{Rosser theory  for $n$-ary relations} ($n\geq 2$) if for any disjoint pair $(M_1^n, M_2^n)$ of $n$-ary RE relations, there exists a formula $\phi(x_1, \cdots, x_n)$ with exactly $n$-free variables such that for any $\overrightarrow{a}=(a_1, \cdots, a_n)\in \mathbb{N}^n$, if $\overrightarrow{a} \in M_1^n$, then $T\vdash \phi(\overline{a_1}, \cdots, \overline{a_n})$, and if $\overrightarrow{a} \in M_2^n$, then $T\vdash \neg\phi(\overline{a_1}, \cdots, \overline{a_n})$.
      \item We say $T$ is an  \emph{exact Rosser theory  for RE sets}  if for any disjoint pair $(A,B)$ of RE sets, there exists a formula $\phi(x)$ with exactly one free variable such that $n\in A\Leftrightarrow T\vdash \phi(\overline{n})$, and  $n\in B\Leftrightarrow T\vdash \neg\phi(\overline{n})$. Similarly, we can define the notion of \emph{exact Rosser theory  for $n$-ary relations}  ($n\geq 2$). 
      \item  We denote the RE set with index $i$ by $W_i$. We say $T$ is \emph{effectively Rosser for RE sets} if there exists a recursive function $f(i,j)$ such that for any $i,j\in\omega$, $f(i,j)$ is the G\"{o}del number of a formula $\phi(x)$ such that for any $n\in\omega$, if $n\in W_i-W_j$, then $T\vdash \phi(\overline{n})$ and if $n\in W_j-W_i$, then $T\vdash \neg\phi(\overline{n})$.  
          \item We say $T$ is \emph{effectively exact Rosser for RE sets} if there exists a recursive function $f(i,j)$ such that for any disjoint pair of RE sets $(W_i, W_j)$, $f(i,j)$ is the G\"{o}del number of a formula $\phi(x)$ such that for any $n\in\omega$, $n\in W_i\Leftrightarrow T\vdash \phi(\overline{n})$ and $n\in W_j\Leftrightarrow T\vdash \neg\phi(\overline{n})$. 
   \item    We say $T$ is \emph{Rosser} if $T$ is a Rosser theory for RE sets and a Rosser theory for $n$-ary relations for any $n\geq 2$.
\end{enumerate}
\end{definition}

The first known definitions of Rosser theories appear in \cite{Smullyan59}. In \cite{Smullyan59}, Rosser theories are defined as Rosser theories  for RE sets.

Results in \cite{Smullyan93} are mostly about Rosser theories for RE sets, and very few theorems in \cite{Smullyan93} are about Rosser theories for 2-ary relations. Properties of Rosser theories for $n$-ary relations and relationships among them  are not discussed in \cite{Smullyan93}.
Based on Definition \ref{standard def of Rosser}, Theorem \ref{effect Rosser} and Theorem \ref{PST} are proved in \cite{Smullyan93}:
\begin{theorem}[\cite{Smullyan93}]~\label{effect Rosser}
\begin{enumerate}[(1)]
  \item If $T$ is Rosser for binary RE relations, then $T$ is effectively Rosser for RE sets;
  \item If $T$ is exact Rosser for binary RE relations, then $T$ is effectively exact Rosser for RE sets.
      \item If $T$ is Rosser for binary RE relations, then $T$ is exact Rosser for RE sets.
      \item  A theory $T$ is effectively Rosser for RE sets if and only if $T$ is effectively exact Rosser for RE sets.
\end{enumerate}
\end{theorem}

\begin{theorem}[Putnam-Smullyan Theorem,\cite{Smullyan93}]~\label{PST}
Suppose $T$ is Rosser for RE sets and any $1$-ary recursive function is strongly definable in $T$.\footnote{We say a function $f(x)$ is strongly definable in $T$ if there exists a formula $\varphi(x,y)$ such that for any $n\in\omega, T\vdash \forall y [\varphi(\overline{n},y)\leftrightarrow y=\overline{f(n)}]$.} Then $T$ is exact Rosser for RE sets.
\end{theorem}

\subsection{A new definition of Rosser theories}\label{new def}

In this section, we give new definitions of $n$-Rosser theories, exact $n$-Rosser theories, effectively $n$-Rosser theories and effectively exact $n$-Rosser theories. Based on the notion of $n$-Rosser theories, we define that $T$ is Rosser if $T$ is $n$-Rosser for any $n\geq 1$. In these new definitions, the language of the base theory is not restricted to arithmetic languages (or is not required to admit numerals for natural numbers).  Instead, we only require that numerals of natural numbers are interpretable in the base theory (for the definition of interpretation, we refer to Definition \ref{def of interpretation}). For a theory $T$ whose language does not admit numerals, to make sure that we can talk about ``numerals" in $T$, our strategy is to propose a simple and natural theory of numerals and require that this theory of numerals is interpretable in  $T$. There are varied choices of a theory of numerals. The reason for our choice of the theory $\sf Num$ in Definition \ref{theory of numerals} is due to its simplicity and naturalness for us.

\begin{definition}\label{theory of numerals}
Let $\sf Num$ denote the theory in the language $\{\mathbf{0},\mathbf{S}\}$ with the following axiom scheme: $\overline{m}\neq \overline{n}$ if $m\neq n$, where $\overline{n}$ is defined recursively as $\overline{0}=\mathbf{0}$ and $\overline{n+1}=\mathbf{S}\overline{n}$.
\end{definition}

Now we introduce  $n$-Rosser theories, exact $n$-Rosser theories, effectively $n$-Rosser theories and effectively exact $n$-Rosser theories  in a general setting.

\begin{definition}\label{}
Let $T$ be a consistent RE theory.  Suppose $I: {\sf Num} \unlhd T$, $\phi(x_1, \cdots, x_n)$ is a formula with $n$-free variables, $M_1^n$ and $M_2^n$ are two $n$-ary RE relations.\footnote{For the definition of the notation $S\unlhd T$, we refer to Definition \ref{def of interpretation}.}
\begin{enumerate}[(1)]
  \item We say $\phi(x_1, \cdots, x_n)$ \emph{strongly separates} $M_1^n-M_2^n$ from $M_2^n-M_1^n$ in $T$ with respect to (w.r.t. for short) $I$ if   $(a_1, \cdots, a_n)\in M_1^n- M_2^n\Rightarrow T\vdash \phi(\overline{a_1}^I, \cdots, \overline{a_n}^I)$, and  $(a_1, \cdots, a_n)\in M_2^n- M_1^n\Rightarrow T\vdash \neg\phi(\overline{a_1}^I, \cdots, \overline{a_n}^I)$.
  \item
Suppose $M_1^n$ and $M_2^n$ are disjoint. We say $\phi(x_1, \cdots, x_n)$ \emph{exactly separates} $M_1^n$ from $M_2^n$ in $T$ w.r.t. $I$ if $(a_1, \cdots, a_n)\in M_1^n\Leftrightarrow T\vdash \phi(\overline{a_1}^I, \cdots, \overline{a_n}^I)$, and  $(a_1, \cdots, a_n)\in M_2^n\Leftrightarrow T\vdash \neg\phi(\overline{a_1}^I, \cdots, \overline{a_n}^I)$.
\end{enumerate}
\end{definition}

Let $\langle R^n_0, \cdots, R^n_i, \cdots\rangle$ be an acceptable listing of all $n$-ary RE relations. We always assume that $R^n_i$  is a $n$-ary RE relation with index $i$. In this paper, both $\overrightarrow{x}\in R^n_i$ and $R^n_i(\overrightarrow{x})$ mean that $R^n_i$ holds for $\overrightarrow{x}$.

\begin{definition}\label{def of Rosser}
Let $T$ be a consistent RE theory and $n\geq 1$.
\begin{enumerate}[(1)]
  \item We say $T$ is \emph{$n$-Rosser} if there exists an interpretation $I: {\sf Num} \unlhd T$ such that for any pair of $n$-ary RE relations $M_1^n$ and $M_2^n$, there is a formula $\phi(x_1, \cdots, x_n)$ with exactly $n$-free variables such that $\phi(x_1, \cdots, x_n)$ strongly separates $M_1^n-M_2^n$ from $M_2^n-M_1^n$ in $T$ w.r.t. $I$.
  \item We say $T$ is \emph{exact $n$-Rosser} if there exists an interpretation $I: {\sf Num} \unlhd T$ such that for any disjoint pair of $n$-ary RE relations $M_1^n$ and $M_2^n$, there is a formula $\phi(x_1, \cdots, x_n)$ with exactly $n$-free variables such that  $\phi(x_1, \cdots, x_n)$ exactly separates $M_1^n$ from $M_2^n$ in $T$ w.r.t. $I$.
\item
We say $T$ is \emph{effectively $n$-Rosser} if there exists an interpretation $I: {\sf Num} \unlhd T$ and a recursive function $f(i,j)$ such that for any pair of $n$-ary RE  relations  $R_i^n$ and $R_j^n$, $f(i,j)$ is the code of a formula $\phi(x_1, \cdots, x_n)$ with exactly $n$-free variables such that $\phi(x_1, \cdots, x_n)$ strongly separates $R_i^n-R_j^n$ from $R_j^n-R_i^n$ in $T$ w.r.t. $I$.
\item We say $T$ is \emph{effectively exact $n$-Rosser} if there exists an interpretation $I: {\sf Num} \unlhd T$ and a recursive function $f(i,j)$ such that for any pair of disjoint $n$-ary RE relations $R_i^n$ and $R_j^n$, $f(i,j)$ is the code of a formula $\phi(x_1, \cdots, x_n)$ with exactly $n$-free variables which exactly separates $R_i^n$ from $R_j^n$ in $T$ w.r.t. $I$.
    \item If the theory $T$ is a relational extension of {\sf Num}, we assume that the interpretation $I$ in above definitions is just the identity function.\footnote{I.e., for a given $I$, if it is based on a relational expansion of numerals, we just take it to be the identity function.}
\end{enumerate}
\end{definition}

\begin{definition}\label{general def}
Let $T$ be a consistent RE theory.
\begin{enumerate}[(1)]
  \item We say $T$ is \emph{Rosser} if for any $n\geq 1$, $T$ is $n$-Rosser.
  \item We say $T$ is \emph{exact Rosser} if for any $n\geq 1$, $T$ is exact $n$-Rosser.
  \item We say $T$ is \emph{effectively Rosser} if for any $n\geq 1$, $T$ is effectively $n$-Rosser.
  \item We say $T$ is \emph{effectively exact Rosser} if for any $n\geq 1$, $T$ is effectively exact $n$-Rosser.
\end{enumerate}
\end{definition}

One referee commented that each definition in Definition \ref{general def} has a local version and a global version (which is stronger): the local version allows the witnessing interpretation function $I$ to vary with $n$; and for the global version, there is a fixed interpretation function $I$ that works for all $n$. Definition \ref{general def} is the local version, and we did not explore the global version in this paper. In section \ref{generalization of theorem on rosser}, we show that all the four notions in Definition \ref{general def} are equivalent. If we formulate Definition \ref{general def} via the global version, from the proof of Theorem \ref{equivalent thm}, we can see that the four notions in the global version are also  equivalent.  

In the following, we aim to study  properties of $n$-Rosser theories, exact $n$-Rosser theories, effectively $n$-Rosser theories and effectively exact $n$-Rosser theories, and relationships among them. Especially, we will generalize Theorem \ref{effect Rosser} and Theorem \ref{PST} to $n$-Rosser theories for any $n\geq 1$.

\section{Preliminary}\label{prelim}

In this section, we list definitions and facts we will use later.
We always assume that $T$ is a consistent RE theory in a language.
Let $\langle W_i: i\in\omega\rangle$ be the list of all RE sets.
For any $x \in \mathbb{N}$, define $\overbrace{x}=(x, \cdots, x)\in \mathbb{N}^n$. The length of the vector $\overbrace{x}$ will be clear from the context.
We use  $E_n(x_1, \cdots, x_m)$ to denote a formula with code number $n$ whose free variables are among  $x_1, \cdots, x_m$.

Let $J_2(x,y)$ be the paring function, and $K(x)$ and $L(x)$ be the recursive functions such that for any $x,y\in\omega$, we have $K(J_2(x,y))=x$ and $L(J_2(x,y))=y$.
We can define the recursive $(n+1)$-ary pairing function as follow:
\[J_{n+1}(x_1, \cdots, x_{n+1})\triangleq J_2(J_{n}(x_1, \cdots, x_{n}),x_{n+1}).\]

Now we introduce the notion of interpretation.
\begin{definition}[Translations and interpretations, \cite{Visser16}, pp.10-13]~\label{def of interpretation}
\begin{itemize}
\item We use $L(T)$ to denote the language of the theory $T$. Let $T$ be a theory in a language $L(T)$, and $S$ a theory in a language $L(S)$. In its simplest form, a \emph{translation} $I$ of language $L(T)$ into language $L(S)$ is specified by the following:
\begin{itemize}
  \item an $L(S)$-formula $\delta_I(x)$ denoting the domain of $I$;
  \item for each relation symbol $R$ of $L(T)$,  as well as the equality relation =, an $L(S)$-formula $R_I$ of the same arity;
  \item for each function symbol $F$ of $L(T)$ of arity $k$, an $L(S)$-formula $F_I$ of arity $k + 1$.
\end{itemize}
\item If $\phi$ is an $L(T)$-formula, its \emph{$I$-translation} $\phi^I$ is an $L(S)$-formula constructed as follows: we rewrite the
formula in an equivalent way so that function symbols only occur in atomic subformulas of the
form $F(\overline{x}) = y$, where $\overline{x}, y$ are variables; then we replace each such atomic formula with $F_I(\overline{x}, y)$,
we replace each atomic formula of the form $R(\overline{x})$ with $R_I(\overline{x})$, and we restrict all quantifiers and
free variables to objects satisfying $\delta_I$. We take care to rename bound variables to avoid variable
capture during the process.
\item A translation $I$ of $L(T)$ into $L(S)$ is an \emph{interpretation} of $T$ in $S$ if $S$ proves the following:
\begin{itemize}
  \item for each function symbol $F$ of $L(T)$ of arity $k$, the formula expressing that $F_I$ is total on $\delta_I$:
\[\forall x_0, \cdots \forall x_{k-1} (\delta_I(x_0) \wedge \cdots \wedge \delta_I(x_{k-1}) \rightarrow \exists y (\delta_I(y) \wedge F_I(x_0, \cdots, x_{k-1}, y)));\]
  \item the $I$-translations of all theorems of $T$, and axioms of equality.
\end{itemize}
\item A theory $T$ is \emph{interpretable} in a theory $S$ if there exists an
interpretation of $T$ in $S$.
\item Given theories $T$ and $S$, let $I: T\unlhd S$ denote that $T$ is interpretable in $S$ (or $S$ interprets $T$) via an interpretation $I$.
\end{itemize}
\end{definition}

The theory $\mathbf{R}$ introduced in \cite{Tarski} is important in the study of meta-mathematics of arithmetic.

\begin{definition}\label{def of R}
Let $\mathbf{R}$ be the theory consisting of the following axiom schemes with signature  $\{\mathbf{0}, \mathbf{S}, +, \cdot\}$ where $x\leq y\triangleq \exists z (z+x=y)$.
\begin{description}
  \item[\sf{Ax1}] $\overline{m}+\overline{n}=\overline{m+n}$;
  \item[\sf{Ax2}] $\overline{m}\cdot\overline{n}=\overline{m\cdot n}$;
  \item[\sf{Ax3}] $\overline{m}\neq\overline{n}$, if $m\neq n$;
  \item[\sf{Ax4}] $\forall x(x\leq \overline{n}\rightarrow x=\overline{0}\vee \cdots \vee x=\overline{n})$;
  \item[\sf{Ax5}] $\forall x(x\leq \overline{n}\vee \overline{n}\leq x)$.
\end{description}
\end{definition}

\begin{lemma}[Separation Lemma, \cite{Smullyan93}]\label{}
For any RE sets $A$ and $B$, there exist RE sets $C$ and $D$ such that $A-B\subseteq C, B-A\subseteq D, C\cap D=\emptyset$ and $A\cup B=C\cup D$.
\end{lemma}

By the separation lemma, the notion that $T$ is $n$-Rosser is equivalent with: there exists an interpretation $I: {\sf Num} \unlhd T$ such that for any disjoint $n$-ary RE relations $M_1^n$ and $M_2^n$, there exists a formula $\phi(x_1, \cdots, x_n)$ with exactly $n$-free variables such that $\phi(x_1, \cdots, x_n)$ strongly separates $M_1^n$ from $M_2^n$ in $T$ w.r.t. $I$. Similarly,  the notion that $T$ is effectively $n$-Rosser is equivalent with the version in which we can assume that the $n$-ary RE relations $M_1^n$ and $M_2^n$ are disjoint.

%\begin{theorem}[s-m-n Theorem, Theorem 2 in p.52, \cite{Smullyan93}]\label{}
%For any $m,n>0$ and any RE relation $M(x_1, \cdots, x_m, y_1, \cdots, y_n)$, there exists a recursive function $f(y_1, \cdots, y_n)$ such that for any $x_1, \cdots, x_m, y_1, \cdots, y_n\in\omega$, we have:
%\[(x_1, \cdots, x_m)\in R^m_{f(y_1, \cdots, y_n)}\Leftrightarrow M(x_1, \cdots, x_m, y_1, \cdots, y_n).\]
%\end{theorem}

\begin{definition}\label{}
For any $m, n\in \mathbb{N}$, we call a function $F: \mathbb{N}^m\rightarrow \mathbb{N}^n$ a \emph{$n$-ary functional on $\mathbb{N}^m$}.\footnote{Given a  $n$-ary functional $F: \mathbb{N}^m\rightarrow \mathbb{N}^n$ on $\mathbb{N}^m$, it naturally induces $m$-ary function $f_i: \mathbb{N}^m\rightarrow \mathbb{N}$ such that for any $\overrightarrow{a}\in \mathbb{N}^m$, $f_i(\overrightarrow{a})=F(\overrightarrow{a})_{i-1}$ for $1\leq i\leq n$.
Given $m$-ary functions $f_i: \mathbb{N}^m\rightarrow \mathbb{N}$ for $1\leq i\leq n$, we can naturally define a $n$-ary functional $F: \mathbb{N}^m\rightarrow \mathbb{N}^n$ on $\mathbb{N}^m$ such that for any $\overrightarrow{a}\in \mathbb{N}^m, F(\overrightarrow{a})=(f_1(\overrightarrow{a}), \cdots, f_n(\overrightarrow{a}))$.}
\end{definition}

\begin{convention}\label{}
Since there is a one-to-one correspondence between  a $n$-ary functional $F: \mathbb{N}^m\rightarrow \mathbb{N}^n$ on $\mathbb{N}^m$ and a sequence $(f_1(\overrightarrow{x}), \cdots, f_n(\overrightarrow{x}))$ of $m$-ary functions with length $n$, throughout this paper, we write a $n$-ary functional $F(\overrightarrow{x})$ on $\mathbb{N}^m$ as $F(\overrightarrow{x})=(f_1(\overrightarrow{x}), \cdots, f_n(\overrightarrow{x}))$.
\end{convention}

\begin{definition}\label{}
We say $F(\overrightarrow{x})=(f_1(\overrightarrow{x}), \cdots, f_n(\overrightarrow{x}))$ on $\mathbb{N}^m$ is a \emph{recursive $n$-ary functional}   if for any $1\leq i\leq n$, $f_i$ is a recursive $m$-ary function.
\end{definition}

We will use the s-m-n theorem throughout this paper and we refer it to \cite[Theorem 2, p.52]{Smullyan93}.

\section{Some basic facts about Rosser theories}\label{basic facts}

In this section, we prove some basic facts about Rosser theories.
Fact \ref{basic fact on n-Rosser} is an easy observation about relationships among notions in Definition \ref{def of Rosser}.

\begin{fact}\label{basic fact on n-Rosser}
\begin{enumerate}[(1)]
  \item For any $m> n$, if $T$ is $m$-Rosser, then $T$ is $n$-Rosser. As a corollary, if $T$ is $n$-Rosser for $n\geq 2$, then  $T$ is $1$-Rosser;
  \item Exact $n$-Rosser implies $n$-Rosser;
  \item Effectively $n$-Rosser implies $n$-Rosser;
  \item Effectively exact $n$-Rosser implies effectively $n$-Rosser;
\item  Effectively exact $n$-Rosser implies exact $n$-Rosser.
\end{enumerate}
\end{fact}
\begin{proof}\label{}
We only prove (1): the other claims are trivial. Suppose $m> n$ and $T$ is $m$-Rosser under $I: {\sf Num}\unlhd T$. We show that $T$ is $n$-Rosser. Suppose $M^n_1$ and $M^n_2$ are two disjoint $n$-ary RE relations. Define two $m$-ary RE relations $S^m_1$ and $S^m_2$ such that for any $\overrightarrow{a}=(a_1, \cdots, a_m)\in \mathbb{N}^m$, $(a_1, \cdots, a_m)\in S^m_1\Leftrightarrow (a_1, \cdots, a_n)\in M^n_1$ and $(a_1, \cdots, a_m)\in S^m_2\Leftrightarrow (a_1, \cdots, a_n)\in M^n_2$. Since $T$ is $m$-Rosser, there is a formula $\phi(x_1, \cdots, x_m)$ with $m$-free variables which strongly separates $S^m_1$ from $S^m_2$. Define $\psi(x_1, \cdots, x_n)\triangleq \phi(x_1, \cdots, x_n, \overline{0}^I,\cdots, \overline{0}^I)$. Then $\psi(x_1, \cdots, x_n)$ strongly separates $M^n_1$ from $M^n_2$.
\end{proof}

\begin{definition}\label{}
We say that a consistent RE theory $T$ is \emph{essentially Rosser} if any consistent RE extension of $T$ is also Rosser.
\end{definition}

\begin{proposition}
A theory $T$ is Rosser if and only if $T$ is essentially Rosser.
\end{proposition}
\begin{proof}\label{}
This follows from the fact: if $T$ is Rosser and $S$ is a consistent RE extension of $T$, then $S$ is Rosser.
\end{proof}

\begin{theorem}\label{pre thm on inter}
If $T$ is Rosser and $T$ is interpretable in $S$, then $S$ is Rosser.
\end{theorem}
\begin{proof}\label{}
It suffices to show that for any $n\geq 1$, if $T$ is $n$-Rosser and $T$ is interpretable in $S$, then $S$ is $n$-Rosser.
Suppose $I$ is the witness interpretation for $T$ being $n$-Rosser, and $J$ is the witness interpretation for $T$ being interpretable in $S$. Define $K= J\circ I$.
We show that  $K$  is the witness interpretation for $S$ being $n$-Rosser.
Suppose $M^n_1$ and $M^n_2$ are two disjoint $n$-ary RE relation. There exists $\phi(v_1, \cdots, v_n)$ such that for any $\overrightarrow{a}=(a_1, \cdots, a_n)\in \mathbb{N}^n$,
\begin{align*}
\overrightarrow{a}\in M^n_1\Rightarrow T\vdash  \phi(\overline{a_1}^I, \cdots, \overline{a_n}^I)\Rightarrow S\vdash \phi^{J}((\overline{a_1}^I)^J, \cdots, (\overline{a_n}^I)^J); \\
\overrightarrow{a}\in M^n_2\Rightarrow T\vdash  \neg\phi(\overline{a_1}^I, \cdots, \overline{a_n}^I)\Rightarrow S\vdash \neg\phi^{J}((\overline{a_1}^I)^J, \cdots, (\overline{a_n}^I)^J).
\end{align*}
Note that $(\overline{a_i}^I)^J=\overline{a_i}^K$ for $1\leq i\leq n$. Define $\psi(\overrightarrow{x})=\phi^{J}(\overrightarrow{x})$. Then $\overrightarrow{a}\in M^n_1\Rightarrow S\vdash  \psi(\overline{a_1}^K, \cdots, \overline{a_n}^K)$ and
$\overrightarrow{a}\in M^n_2\Rightarrow S\vdash  \neg\psi(\overline{a_1}^K, \cdots, \overline{a_n}^K)$. Thus, $\psi(\overrightarrow{x})$ strongly separates $M^n_1$ from $M^n_2$ in $S$ w.r.t. $K$.
\end{proof}

A natural question is: is there any natural example of Rosser theories? The theory  $\mathbf{R}$ is  such a natural example. Moreover, from results in this paper, we can see that the theory $\mathbf{R}$ has all properties we have introduced in this paper. 

\begin{theorem}\label{R is Rosser}
The theory  $\mathbf{R}$ is  Rosser.
\end{theorem}
\begin{proof}\label{}
It is a well known fact that the theory $\mathbf{R}$ is Rosser  for RE sets (see \cite{Smullyan93}). The proof that $\mathbf{R}$ is  Rosser under Definition \ref{general def} is a straightforward
generalization to more variables, of the classical proof that $\mathbf{R}$ is Rosser  for RE sets.
\end{proof}

\begin{corollary}\label{}
Both $\mathbf{PA}$ and $\mathbf{ZFC}$ are Rosser.
\end{corollary}
\begin{proof}\label{}
Follows from Theorem \ref{R is Rosser} and Theorem \ref{pre thm on inter}.
\end{proof}

In the rest of this section, we examine the relationship between Rosser theories and effectively inseparable theories. 

\begin{definition}[The nuclei of a theory, {\sf EI} theories]~\label{def of 1.1}
Let $T$ be a consistent RE theory, and $(A,B)$ be a disjoint pair of RE sets.
\begin{enumerate}[(1)]
\item A pair $(A,B)$ of disjoint RE sets is  \emph{effectively inseparable} $(\sf EI)$ if there is a recursive function $f(x,y)$ such that for any $i, j\in\omega$, if $A\subseteq W_i, B\subseteq W_j$ and $W_i\cap W_j=\emptyset$, then $f(i,j)\notin W_i\cup W_j$.
    \item The pair $(T_P, T_R)$ is called the \emph{nuclei} of a theory $T$, where  $T_P$ is the set of G\"{o}del numbers of sentences provable in $T$, and $T_R$ is the set of G\"{o}del numbers of sentences refutable in $T$ (i.e., $T_P=\{\ulcorner\phi\urcorner: T\vdash\phi\}$ and $T_R=\{\ulcorner\phi\urcorner: T\vdash\neg\phi\}$).
        \item We say $T$ is \emph{effectively inseparable} $(\sf EI)$ if $(T_P, T_R)$ is $\sf EI$.
\end{enumerate}
\end{definition}

\begin{theorem}[\cite{Smullyan93}, pp.70-126]~\label{EI thm}
For any consistent RE theory $T$, $T$ is {\sf EI} iff for any disjoint pair $(A,B)$ of RE sets, there is a recursive function $f(x)$ such that if $x\in A$, then $f(x)\in T_P$, and if $x\in B$, then $f(x)\in T_R$.
\end{theorem}

\begin{theorem}\label{}
For any $n\geq 1$, if $T$ is $n$-Rosser, then $T$ is $\sf EI$. Thus, if $T$  is Rosser, then $T$ is $\sf EI$.
\end{theorem}
\begin{proof}\label{}
By Fact \ref{basic fact on n-Rosser}, it suffices to show that if $T$ is 1-Rosser, then $T$ is $\sf EI$.
Suppose $T$ is 1-Rosser via the interpretation $I: {\sf Num}\unlhd T$. By Theorem \ref{EI thm}, to show that $T$ is $\sf EI$, it suffices to show that for any disjoint RE pair $(A,B)$, there is a recursive function $f$ such that if $n\in A$, then $f(n)\in T_P$; and if $n\in B$, then $f(n)\in T_R$.

Let $(A,B)$ be a disjoint pair of RE sets. Then there is a formula $\phi(x)$ such that if $n\in A$, then $T\vdash \phi(\overline{n}^I)$, and if $n\in B$, then $T\vdash \neg\phi(\overline{n}^I)$.
Define $f(n)\triangleq\ulcorner\phi(\overline{n}^I)\urcorner$. Then $n\in A\Rightarrow f(n)\in T_P$ and  $n\in B\Rightarrow f(n)\in T_R$.
\end{proof}

%\begin{theorem}\label{}
%For any $n\geq 1$, if $S$ and $T$ are $n$-Rosser theories, then $S\oplus T$ is also %$n$-Rosser.
%\end{theorem}

%\begin{theorem}
%For any $n\geq 1$, $T$ is $1$-Rosser does not imply $T$ interprets $\mathbf{R}$.
%\end{theorem}

\begin{theorem}\label{}
For any $n\geq 1$, $\sf EI$ does not imply $n$-Rosser. Thus, $\sf EI$ does not imply Rosser.
\end{theorem}
\begin{proof}\label{}
By Fact \ref{basic fact on n-Rosser}, it suffices to show that $\sf EI$ does not imply 1-Rosser.
Let $\sf Succ$ be the theory over the language $\{\mathbf{0},\mathbf{S}\}$ consisting of axioms $\sf S1, \sf S2$ and $\sf S3$.
\begin{description}
  \item[$\sf S1$] $\forall x \forall y(\mathbf{S} x=\mathbf{S} y\rightarrow x=y)$;
  \item[$\sf S2$] $\forall x(\mathbf{S} x\neq \mathbf{0})$;
  \item[$\sf S3$] $\forall x(x\neq \mathbf{0}\rightarrow \exists y (x=\mathbf{S} y))$.
\end{description}
By Theorem 4.4 in \cite{Cheng23}, for any  consistent extension $S$ of $\sf Succ$ over the same language and $X\subseteq \mathbb{N}$, $X$ is weakly representable\footnote{We say $X\subseteq \mathbb{N}$ is weakly representable in $S$ if there exists a formula $\phi(x)$ such that for any $n\in \omega$, $n\in X\Leftrightarrow S\vdash \phi(\overline{n})$.} in $S$ iff $X$ is finite or co-finite.

We work in the language $\{\mathbf{0},\mathbf{S}\}$. Define the sentence $\phi_n\triangleq \exists x (\mathbf{S}^{n}x=x)$, and the  theory $T\triangleq \sf Succ + \{\phi_n: n\in B\} + \{\neg \phi_n: n\in C\}$ where $(B,C)$ is an $\sf EI$ pair.
By a standard argument, we can show that $T$ is $\sf EI$ (see Theorem 3.12 in \cite{Cheng23}).

Suppose $T$ is  1-Rosser. Then any recursive set is weakly representable in $T$.\footnote{Suppose $T$ is $1$-Rosser and $A$ is a recursive set.  Since $T$ extends {\sf Num}, by Definition \ref{def of Rosser}, the witness interpretation function for $T$'s being 1-Rosser is just the identity function. Then there is a formula $\phi(x)$ with exactly one free variable such that $\phi(x)$ strongly separates $A$ from the complement of $A$ in $T$. I.e., if $n\in A$, then $T\vdash \phi(\overline{n})$, and if $n\notin A$, then $T\vdash \neg\phi(\overline{n})$. Thus,  $n\in A\leftrightarrow T\vdash \phi(\overline{n})$. So $\phi(x)$ weakly represents $A$ in $T$.}
But any set weakly representable in $T$ is finite or co-finite. Thus, any recursive set is finite or co-finite, which leads to a contradiction. Thus, $T$ is not 1-Rosser.
\end{proof}

\section{A generalization of the Strong Double Recursion Theorem}\label{SRDT}

In \cite{Smullyan93}, Smullyan proved the Strong Double Recursion  as in Theorem \ref{original SDRT}. In this section, we propose a generalized version of $\sf SDRT$ as in Theorem \ref{SDRT}. We will apply Theorem \ref{SDRT} to generalize Theorem \ref{effect Rosser} and Theorem \ref{PST} to $n$-Rosser theories.

\begin{theorem}[The Strong Double Recursion Theorem ($\sf SDRT$),\cite{Smullyan93}]~\label{original SDRT}
For any RE relations $M_1(x, y_1, y_2, z_1,z_2)$ and  $M_2(x, y_1, y_2, z_1,z_2)$, there are recursive functions $t_1(y_1, y_2)$ and $t_2(y_1, y_2)$ such that for any $y_1, y_2\in\omega$,
\begin{enumerate}[(1)]
  \item $x\in W_{t_1(y_1, y_2)}\Leftrightarrow M_1(x, y_1, y_2, t_1(y_1, y_2), t_2(y_1, y_2))$;
  \item $x\in W_{t_2(y_1, y_2)}\Leftrightarrow M_2(x, y_1, y_2, t_1(y_1, y_2), t_2(y_1, y_2))$.
\end{enumerate}
\end{theorem}

\begin{theorem}\label{SDRT}
Let $M_1(\overrightarrow{x}, \overrightarrow{y_1}, \overrightarrow{y_2}, z_1,z_2)$ and  $M_2(\overrightarrow{x}, \overrightarrow{y_1}, \overrightarrow{y_2}, z_1,z_2)$ be two  $(n+2m+2)$-ary RE relations. Then there are $2m$-ary recursive functions $t_1(\overrightarrow{y_1}, \overrightarrow{y_2})$ and $t_2(\overrightarrow{y_1}, \overrightarrow{y_2})$ such that for any $\overrightarrow{y_1}, \overrightarrow{y_2}\in \mathbb{N}^m$,
\begin{enumerate}[(1)]
  \item $\overrightarrow{x}\in R^n_{t_1(\overrightarrow{y_1}, \overrightarrow{y_2})}\Leftrightarrow M_1(\overrightarrow{x}, \overrightarrow{y_1}, \overrightarrow{y_2}, t_1(\overrightarrow{y_1}, \overrightarrow{y_2}), t_2(\overrightarrow{y_1}, \overrightarrow{y_2}))$;
  \item $\overrightarrow{x}\in R^n_{t_2(\overrightarrow{y_1}, \overrightarrow{y_2})}\Leftrightarrow M_2(\overrightarrow{x}, \overrightarrow{y_1}, \overrightarrow{y_2}, t_1(\overrightarrow{y_1}, \overrightarrow{y_2}), t_2(\overrightarrow{y_1}, \overrightarrow{y_2}))$.
\end{enumerate}
\end{theorem}
The proof of Theorem \ref{SDRT} is a straightforward modification of the proof of Theorem \ref{original SDRT} in \cite{Smullyan93}, replacing $x, y_1, y_2$ with vectors $\overrightarrow{x}, \overrightarrow{y_1}, \overrightarrow{y_2}$. For completeness, we include a proof of Theorem \ref{SDRT} in Appendix \ref{1st proof}. 

One referee correctly points out that Theorem \ref{original SDRT} and Theorem \ref{SDRT} are generalizations of the Double Recursion Theorem with parameters in \cite{Rogers87}. In recursion theory, even if the Double Recursion Theorem with parameters can be viewed as a natural generalization of the Recursion Theorem with parameters,  the Double Recursion Theorem with parameters indeed provides us with a powerful tool for discovering more new conclusions in applications. Similarly, even if Theorem \ref{SDRT} is an obvious generalization of Theorem \ref{original SDRT}, Theorem \ref{SDRT} is a powerful and useful tool for generalizing results about Rosser theories for RE sets to results about the hierarchy of $n$-Rosser theories. 

Theorem \ref{coro of SDRT} is a corollary of Theorem \ref{SDRT} which we will use later.

\begin{theorem}\label{coro of SDRT}
For any $3n$-ary RE relations $M_1(\overrightarrow{x}, \overrightarrow{y}, \overrightarrow{z})$ and $M_2(\overrightarrow{x}, \overrightarrow{y}, \overrightarrow{z})$, for any recursive functional $G(x,y)$ on $\mathbb{N}^2$, there are recursive $n$-ary functions $f_1(\overrightarrow{y})$ and $f_2(\overrightarrow{y})$ such that for any $\overrightarrow{y}\in \mathbb{N}^n$,
\begin{enumerate}[(1)]
  \item  $\overrightarrow{x}\in R^n_{f_1(\overrightarrow{y})} \Leftrightarrow M_1(\overrightarrow{x}, \overrightarrow{y}, G(f_1(\overrightarrow{y}), f_2(\overrightarrow{y})))$;
  \item $\overrightarrow{x}\in R^n_{f_2(\overrightarrow{y})} \Leftrightarrow M_2(\overrightarrow{x}, \overrightarrow{y}, G(f_1(\overrightarrow{y}), f_2(\overrightarrow{y})))$.
\end{enumerate}
\end{theorem}
\begin{proof}\label{}
Define $M_1^{\ast}(\overrightarrow{x}, \overrightarrow{y_1}, \overrightarrow{y_2}, z_1,z_2)\triangleq M_1(\overrightarrow{x}, \overrightarrow{y_1}, G(z_1,z_2))$
 and  $M_2^{\ast}(\overrightarrow{x}, \overrightarrow{y_1}, \overrightarrow{y_2}, z_1,z_2)\triangleq M_2(\overrightarrow{x}, \overrightarrow{y_2}, G(z_1,z_2))$.

Apply Theorem \ref{SDRT} to $M_1^{\ast}(\overrightarrow{x}, \overrightarrow{y_1}, \overrightarrow{y_2}, z_1,z_2)$ and $M_2^{\ast}(\overrightarrow{x}, \overrightarrow{y_1}, \overrightarrow{y_2}, z_1,z_2)$.  There exist $2n$-ary recursive functions  $t_1(\overrightarrow{y_1}, \overrightarrow{y_2})$ and $t_2(\overrightarrow{y_1}, \overrightarrow{y_2})$ such that:
\begin{align*}
\overrightarrow{x}\in R^n_{t_1(\overrightarrow{y_1}, \overrightarrow{y_2})} \Leftrightarrow
M_1^{\ast}(\overrightarrow{x}, \overrightarrow{y_1}, \overrightarrow{y_2}, t_1(\overrightarrow{y_1}, \overrightarrow{y_2}),t_2(\overrightarrow{y_1}, \overrightarrow{y_2}))\Leftrightarrow M_1(\overrightarrow{x}, \overrightarrow{y_1}, G(t_1(\overrightarrow{y_1}, \overrightarrow{y_2}), t_2(\overrightarrow{y_1}, \overrightarrow{y_2})))\\
\overrightarrow{x}\in R^n_{t_2(\overrightarrow{y_1}, \overrightarrow{y_2})} \Leftrightarrow
M_2^{\ast}(\overrightarrow{x}, \overrightarrow{y_1}, \overrightarrow{y_2}, t_1(\overrightarrow{y_1}, \overrightarrow{y_2}),t_2(\overrightarrow{y_1}, \overrightarrow{y_2}))\Leftrightarrow M_2(\overrightarrow{x}, \overrightarrow{y_2}, G(t_1(\overrightarrow{y_1}, \overrightarrow{y_2}), t_2(\overrightarrow{y_1}, \overrightarrow{y_2}))).
\end{align*}
Define $f_1(\overrightarrow{y})=t_1(\overrightarrow{y}, \overrightarrow{y})$ and $f_2(\overrightarrow{y})=t_2(\overrightarrow{y}, \overrightarrow{y})$. Then we have:
\begin{enumerate}[(1)]
  \item  $\overrightarrow{x}\in R^n_{f_1(\overrightarrow{y})} \Leftrightarrow M_1(\overrightarrow{x}, \overrightarrow{y}, G(f_1(\overrightarrow{y}), f_2(\overrightarrow{y})))$;
  \item $\overrightarrow{x}\in R^n_{f_2(\overrightarrow{y})} \Leftrightarrow M_2(\overrightarrow{x}, \overrightarrow{y}, G(f_1(\overrightarrow{y}), f_2(\overrightarrow{y})))$.
\end{enumerate}
\end{proof}

\section{Generalizations of Theorem \ref{effect Rosser} to $n$-Rosser theories}\label{generalization of theorem on rosser}

In this section, we use the generalized Strong Double Recursion Theorem \ref{SDRT} to generalize Theorem \ref{effect Rosser} to $n$-Rosser theories.  Especially, we prove that the following notions are equivalent: Rosser, Effectively Rosser, Exact Rosser, Effectively exact Rosser.

We first show that ``effectively $n$-Rosser" is equivalent with ``effectively exact $n$-Rosser". Before proving Theorem \ref{effective exact}, we first prove a lemma as follows. 

\begin{lemma}\label{RT1}
For any 2-ary recursive function $f(x,y)$, there exist recursive functions $t_1(x,y)$ and $t_2(x,y)$ such that for any $i,j\in \omega$ and $\overrightarrow{a}\in \mathbb{N}^n$, we have:
\begin{enumerate}[(1)]
  \item $R^n_{t_1(i,j)}(\overrightarrow{a})$ iff $R^{n+1}_{i}(\overrightarrow{a}, f(t_1(i,j), t_2(i,j)))$.
  \item $R^n_{t_2(i,j)}(\overrightarrow{a})$ iff $R^{n+1}_{j}(\overrightarrow{a}, f(t_1(i,j), t_2(i,j)))$.
\end{enumerate}
\end{lemma}
\begin{proof}\label{}
Define \[M_1(\overrightarrow{x}, y_1,y_2,z_1,z_2)\triangleq R^{n+1}_{y_1}(\overrightarrow{x}, f(z_1,z_2))\] and \[M_2(\overrightarrow{x}, y_1,y_2,z_1,z_2)\triangleq R^{n+1}_{y_2}(\overrightarrow{x}, f(z_1,z_2)).\]
Apply Theorem \ref{SDRT} to $M_1(\overrightarrow{x}, y_1,y_2,z_1,z_2)$ and $M_2(\overrightarrow{x}, y_1,y_2,z_1,z_2)$. Then there are recursive functions $t_1(x,y)$ and $t_2(x,y)$ such that for any $i,j\in \omega$ and $\overrightarrow{a}\in \mathbb{N}^n$, we have:
\begin{align*}
R^n_{t_1(i,j)}(\overrightarrow{a})\Leftrightarrow M_1(\overrightarrow{a}, i,j, t_1(i,j),t_2(i,j))\Leftrightarrow R^{n+1}_{i}(\overrightarrow{a}, f(t_1(i,j), t_2(i,j))); \\
R^n_{t_2(i,j)}(\overrightarrow{a})\Leftrightarrow M_2(\overrightarrow{a}, i,j, t_1(i,j),t_2(i,j))\Leftrightarrow R^{n+1}_{j}(\overrightarrow{a}, f(t_1(i,j), t_2(i,j))).
\end{align*}
\end{proof}

\begin{theorem}\label{effective exact}
If $T$ is effectively $n$-Rosser, then $T$ is effectively exact $n$-Rosser.
\end{theorem}
\begin{proof}\label{}
Suppose $T$ is effectively $n$-Rosser under a recursive function $f(i,j)$ and an interpretation $I: {\sf Num}\unlhd T$. Apply Lemma \ref{RT1} to the recursive function $f(i,j)$. Take recursive functions $t_1(x,y)$ and $t_2(x,y)$ as in Lemma \ref{RT1}. Define $h(i,j)=f(t_1(i,j), t_2(i,j))$.

\begin{claim}
For any two $(n+1)$-ary relations $R^{n+1}_i$ and $R^{n+1}_j$, $h(i,j)$ codes a formula with $n$-free variables which strongly separates $\{\overrightarrow{a}\in \mathbb{N}^n: (\overrightarrow{a}, h(i,j))\in R^{n+1}_i-R^{n+1}_j\}$ from $\{\overrightarrow{a}\in \mathbb{N}^n: (\overrightarrow{a}, h(i,j))\in R^{n+1}_j- R^{n+1}_i\}$ in $T$ w.r.t. $I$.
\end{claim}
\begin{proof}\label{}

Suppose  $R^{n+1}_i$ and $R^{n+1}_j$ are two $(n+1)$-ary relations. Note that  $h(i,j)$ codes a formula $\phi(x_1, \cdots, x_n)$ with $n$-free variables which strongly separates $R^n_{t_1(i,j)}-R^n_{t_2(i,j)}$ from $R^n_{t_2(i,j)}-R^n_{t_1(i,j)}$ in $T$ w.r.t. $I$.

We show that $\phi(x_1, \cdots, x_n)$ strongly separates
$\{\overrightarrow{a}\in \mathbb{N}^n: (\overrightarrow{a}, h(i,j))\in R^{n+1}_i-R^{n+1}_j\}$ from $\{\overrightarrow{a}\in \mathbb{N}^n: (\overrightarrow{a}, h(i,j))\in R^{n+1}_j- R^{n+1}_i\}$ in $T$ w.r.t. $I$.

Suppose $(\overrightarrow{a}, h(i,j))\in R^{n+1}_i- R^{n+1}_j$ where $\overrightarrow{a}=(a_1, \cdots, a_n)$. Then by Lemma \ref{RT1}, 
$\overrightarrow{a}\in R^n_{t_1(i,j)}-R^n_{t_2(i,j)}$. Then $T\vdash \phi(\overline{a_1}^I, \cdots, \overline{a_n}^I)$. By a similar argument, we have if $(\overrightarrow{a}, h(i,j))\in R^{n+1}_j-R^{n+1}_i$, then $T\vdash \neg\phi(\overline{a_1}^I, \cdots, \overline{a_n}^I)$.
\end{proof}

Note that there exist recursive functions $s_1(x)$ and $s_2(x)$ such that for any $i, b\in \mathbb{N}$ and $\overrightarrow{a}=(a_1, \cdots, a_n)\in \mathbb{N}^n$, we have
\begin{enumerate}[(1)]
  \item  $R^{n+1}_{s_1(i)}(\overrightarrow{a}, b)\Leftrightarrow [R^{n}_{i}(\overrightarrow{a}) \vee T\vdash \neg E_b(\overline{a_1}^I, \cdots, \overline{a_n}^I)]$
  \item $R^{n+1}_{s_2(j)}(\overrightarrow{a}, b) \Leftrightarrow [R^{n}_{j}(\overrightarrow{a}) \vee T\vdash E_b(\overline{a_1}^I, \cdots, \overline{a_n}^I)]$
\end{enumerate}
Define $g(i,j)=h(s_1(i), s_2(j))$. Let $R^n_i$ and $R^n_j$ be two disjoint $n$-ary RE relations. We show that $g(i,j)$ exactly separates $R^n_i$ and $R^n_j$ in $T$ w.r.t. $I$.

Note that $g(i,j)=h(s_1(i), s_2(j))$ codes a formula
$\phi(x_1, \cdots, x_n)$ with $n$-free variables which strongly separates
$\{\overrightarrow{a}\in \mathbb{N}^n: (\overrightarrow{a}, g(i,j))\in R^{n+1}_{s_1(i)}- R^{n+1}_{s_2(j)}\}$ from $\{\overrightarrow{a}\in \mathbb{N}^n: (\overrightarrow{a}, g(i,j))\in R^{n+1}_{s_2(j)}- R^{n+1}_{s_1(i)}\}$ in $T$ w.r.t. $I$.
Note that $E_{g(i,j)}(x_1, \cdots, x_n)$ is $\phi(x_1, \cdots, x_n)$.
For any $\overrightarrow{a}=(a_1, \cdots, a_n)\in \mathbb{N}^n$, we have:
\begin{enumerate}[(1)]
  \item if $[R^{n}_{i}(\overrightarrow{a}) \vee T\vdash \neg \phi(\overline{a_1}^I, \cdots, \overline{a_n}^I)] \wedge \neg [R^{n}_{j}(\overrightarrow{a}) \vee T\vdash \phi(\overline{a_1}^I, \cdots, \overline{a_n}^I)]$, then $T\vdash \phi(\overline{a_1}^I, \cdots, \overline{a_n}^I)$.
  \item if $[R^{n}_{j}(\overrightarrow{a}) \vee T\vdash \phi(\overline{a_1}^I, \cdots, \overline{a_n}^I)]\wedge\neg [R^{n}_{i}(\overrightarrow{a}) \vee T\vdash \neg \phi(\overline{a_1}^I, \cdots, \overline{a_n}^I)]$, then $T\vdash \neg\phi(\overline{a_1}^I, \cdots, \overline{a_n}^I)$.
\end{enumerate}

Now we show that for any $\overrightarrow{a}=(a_1, \cdots, a_n)$,
$R^n_i(\overrightarrow{a})\Leftrightarrow T\vdash \phi(\overline{a_1}^I, \cdots, \overline{a_n}^I)$ and $R^n_j(\overrightarrow{a})\Leftrightarrow T\vdash \neg\phi(\overline{a_1}^I, \cdots, \overline{a_n}^I)$.

We only show that $R^n_j(\overrightarrow{a})\Leftrightarrow T\vdash \neg\phi(\overline{a_1}^I, \cdots, \overline{a_n}^I)$. By a similar argument, we can show that  $R^n_i(\overrightarrow{a})\Leftrightarrow T\vdash \phi(\overline{a_1}^I, \cdots, \overline{a_n}^I)$.

Suppose $R^n_j(\overrightarrow{a})$ holds. We show that $T\vdash \neg\phi(\overline{a_1}^I, \cdots, \overline{a_n}^I)$. Suppose not, i.e., $T\nvdash \neg\phi(\overline{a_1}^I, \cdots, \overline{a_n}^I)$. Then $[R^{n}_{j}(\overrightarrow{a}) \vee T\vdash \phi(\overline{a_1}^I, \cdots, \overline{a_n}^I)]\wedge\neg [R^{n}_{i}(\overrightarrow{a}) \vee T\vdash \neg \phi(\overline{a_1}^I, \cdots, \overline{a_n}^I)]$ holds. By (2), we have $T\vdash \neg\phi(\overline{a_1}^I, \cdots, \overline{a_n}^I)$, which leads to a contradiction. Thus, $T\vdash \neg\phi(\overline{a_1}^I, \cdots, \overline{a_n}^I)$.

Suppose $T\vdash \neg\phi(\overline{a_1}^I, \cdots, \overline{a_n}^I)$. We show that $R^n_j(\overrightarrow{a})$ holds. Suppose not, i.e., $\neg R^n_j(\overrightarrow{a})$ holds. Then $[R^{n}_{i}(\overrightarrow{a}) \vee T\vdash \neg \phi(\overline{a_1}^I, \cdots, \overline{a_n}^I)] \wedge \neg [R^{n}_{j}(\overrightarrow{a}) \vee T\vdash \phi(\overline{a_1}^I, \cdots, \overline{a_n}^I)]$ holds. By (1), we have $T\vdash \phi(\overline{a_1}^I, \cdots, \overline{a_n}^I)$, which contradicts that $T$ is consistent. Thus, $R^n_j(\overrightarrow{a})$ holds.

Thus, $g(i,j)$ is the code of  the formula $\phi(x_1, \cdots, x_n)$ which exactly separates $R^n_i$ and $R^n_j$ in $T$ w.r.t. $I$. Hence, $T$ is effectively exact $n$-Rosser under $g(i,j)$ and $I$.
\end{proof}

\begin{corollary}\label{effective versus effective exact}
Let $T$ be a consistent RE theory. Then for any $n\geq 1$,  $T$ is effectively $n$-Rosser if and only if $T$ is effectively exact $n$-Rosser.
\end{corollary}
\begin{proof}\label{}
Follows from Theorem \ref{effective exact}.
\end{proof}

\begin{theorem}\label{rosser imply effec rosser}
For any $n\geq 1$, if $T$ is $(n+1)$-Rosser, then $T$ is effectively $n$-Rosser.
\end{theorem}
\begin{proof}\label{}
Suppose $T$ is $(n+1)$-Rosser via the interpretation $I: {\sf Num}\unlhd T$. Define $(n+1)$-ary relations $M^{n+1}_1$ and $M^{n+1}_2$ as: $M^{n+1}_1(a_1, \cdots, a_n, y)\Leftrightarrow (a_1, \cdots, a_n)\in R^n_{Ky}$ and $M^{n+1}_2(a_1, \cdots, a_n, y)\Leftrightarrow (a_1, \cdots, a_n)\in R^n_{Ly}$.
Since $T$ is $(n+1)$-Rosser via $I$, there is a formula $\phi(x_1, \cdots, x_{n+1})$ with $(n+1)$-free variables such that $\phi(x_1, \cdots, x_{n+1})$ strongly separates
$M^{n+1}_1-M^{n+1}_2$ from $M^{n+1}_2-M^{n+1}_1$ in $T$ w.r.t. $I$.

Define $h(i,j)=\ulcorner \phi(x_1, \cdots, x_n, \overline{J(i,j)}^I)\urcorner$. Note that $h$ is recursive. We show that $T$ is effectively $n$-Rosser via $h$ and $I$. Suppose $R_i^n$ and $R_j^n$ are two $n$-ary RE relations. Let $\psi(x_1, \cdots, x_n)\triangleq \phi(x_1, \cdots, x_n, \overline{J(i,j)}^I)$. We show that $\psi(x_1, \cdots, x_n)$ strongly separates $R_i^n-R_j^n$ from $R_j^n-R_i^n$ in $T$ w.r.t. $I$.

Suppose $(a_1, \cdots, a_n)\in R_i^n-R_j^n$. Since $(a_1, \cdots, a_n)\in R^n_{K(J(i,j))}- R^n_{L(J(i,j))}$, we have $T\vdash \phi(\overline{a_1}^I, \cdots, \overline{a_n}^I, \overline{J(i,j)}^I)$, that is $T\vdash\psi(\overline{a_1}^I, \cdots, \overline{a_n}^I)$. By a similar argument, if $(a_1, \cdots, a_n)\in R_j^n-R_i^n$, then $T\vdash\neg\psi(\overline{a_1}^I, \cdots, \overline{a_n}^I)$.
\end{proof}

As a corollary, we have a theory $S$ is Rosser if and only if $S$ is effectively Rosser.

\begin{corollary}\label{}
Let $T$ be a consistent RE theory. Then for any $n\geq 1$, if $T$ is $(n + 1)$-Rosser,
then $T$ is effectively exact $n$-Rosser.
\end{corollary}
\begin{proof}\label{}
Follows from Theorem \ref{effective versus effective exact} and Theorem \ref{rosser imply effec rosser}.
\end{proof}

As corollaries, we have: (1) for any $n\geq 1$, if $T$ is exact $(n+1)$-Rosser, then $T$ is effectively exact $n$-Rosser; (2) a theory $T$ is exact Rosser if and only if $T$ is effectively exact Rosser; (3) if $T$ is $(n+1)$-Rosser, then $T$ is exact $n$-Rosser; (4) a theory $T$ is Rosser if and only if $T$ is exact Rosser. 

In summary, we have the following theorem. 
\begin{theorem}\label{equivalent thm}
The following notions are equivalent:
\begin{enumerate}[(1)]
  \item  Rosser;
  \item Effectively Rosser;
  \item  Exact Rosser;
  \item Effectively exact Rosser.
\end{enumerate}
\end{theorem}

\section{$n$-Rosser + strong definability of $n$-ary recursive functionals implies exact $n$-Rosser}\label{exact Rosser}

In this section, we aim to generalize the Putnam-Smullyan Theorem \ref{PST}  to  $n$-Rosser theories. We first prove Theorem \ref{improve PST} showing that  if $T$  is $n$-Rosser and any $n$-ary recursive functional on $\mathbb{N}^n$ is strongly definable in $T$, then $T$ is exact $n$-Rosser. Then, we prove Theorem \ref{key improve thm} which essentially improves Theorem \ref{improve PST}.

In \cite{Smullyan93}, the notions of semi-$\sf DU$ and $\sf DU$ for a disjoint pair of RE sets are defined. One main tool of Putnam-Smullyan's proof of  Theorem \ref{PST} is the result that semi-$\sf DU$ implies $\sf DU$. To prove this result, Smullyan defined a series of metamathematical notions such as semi-$\sf DU$, $\sf KP$, $\sf CEI$, $\sf EI$, $\sf WEI$, $\sf DG$ and $\sf DU$ (for definitions of these notions, we refer to \cite{Smullyan93}).  In fact, Smullyan  proved in \cite{Smullyan93} that all these notions are equivalent.

A natural question is whether we could define similar notions of semi-$\sf DU$ and $\sf DU$ for a disjoint pair of $n$-ary RE relations and show that they are equivalent. The answer is positive. In Section \ref{semi-DU}, we define the notions of semi-$\sf DU$ and $\sf DU$ for a disjoint pair of $n$-ary RE relations (see Definition \ref{notions about DU}). Then we prove that  semi-$\sf DU$  implies $\sf DU$ for a disjoint pair of $n$-ary RE relations. In Appendix \ref{1st proof} and Appendix \ref{3rd proof}, we give two other proofs of this result. Each proof has its own characteristics and applications in meta-mathematics of arithmetic. In Section \ref{apply of DU}, as an application of ``semi-$\sf DU$ implies $\sf DU$", we first generalize Putnam-Smullyan Theorem \ref{PST} and prove Theorem \ref{improve PST}. Then, we essentially improve Theorem \ref{improve PST} as in Theorem \ref{key improve thm}.

\subsection{Semi-DU implies DU}\label{semi-DU}

Smullyan  introduced the notions of semi-$\sf DU$ and $\sf DU$ for a disjoint pair  of RE sets and proved that semi-$\sf DU$ implies $\sf DU$ in \cite{Smullyan93}.
In this section, we first introduce the notions of semi-$\sf DU$ and $\sf DU$ for a disjoint pair  of $n$-ary RE relations (see Definition \ref{notions about DU}). 
Before proving Theorem \ref{semi-DU is DU}, which is
the main theorem of this section, we give the following definitions.

\begin{definition}\label{notions about DU}
Let $(A,B)$ and $(C,D)$ be  disjoint pairs  of $n$-ary RE relations.
\begin{enumerate}[(1)]
  \item We say a $n$-ary functional $F(\overrightarrow{x})=(f_1(\overrightarrow{x}), \cdots, f_n(\overrightarrow{x}))$ on $\mathbb{N}^n$ is a \emph{semi-reduction} from $(C,D)$ to $(A,B)$ if $F(\overrightarrow{x})$ is recursive and  for any $\overrightarrow{a}\in \mathbb{N}^n$, \begin{enumerate}[(i)]
  \item $\overrightarrow{a}\in C\Rightarrow F(\overrightarrow{a})\in A$;
  \item $\overrightarrow{a}\in D\Rightarrow F(\overrightarrow{a})\in B$.
\end{enumerate}
\item We say a $n$-ary functional $F(\overrightarrow{x})=(f_1(\overrightarrow{x}), \cdots, f_n(\overrightarrow{x}))$ on $\mathbb{N}^n$ is a \emph{reduction} from $(C,D)$ to $(A,B)$ if $F(\overrightarrow{x})$ is recursive and  for any $\overrightarrow{a}\in \mathbb{N}^n$,
\begin{enumerate}[(i)]
  \item $\overrightarrow{a}\in C\Leftrightarrow F(\overrightarrow{a})\in A$;
  \item $\overrightarrow{a}\in D\Leftrightarrow F(\overrightarrow{a})\in B$.
\end{enumerate}
\item We say $(C,D)$ is \emph{semi-reducible (reducible)} to $(A,B)$ if there exists a $n$-ary functional $F(\overrightarrow{x})=(f_1(\overrightarrow{x}), \cdots, f_n(\overrightarrow{x}))$ on $\mathbb{N}^n$ such that it is a semi-reduction (reduction) from $(C,D)$ to $(A,B)$.
\item We say $(A,B)$ is \emph{semi-doubly universal (semi-$\sf DU$)}  if for any  disjoint pair $(C,D)$  of $n$-ary RE relations, there exists a semi-reduction from  $(C,D)$ to $(A,B)$.
\item We say $(A,B)$ is \emph{doubly universal ($\sf DU$)} if for any  disjoint pair $(C,D)$  of $n$-ary RE relations,  there exists a reduction from  $(C,D)$ to $(A,B)$.
\end{enumerate}
\end{definition}

Note that $F(\overrightarrow{x})=(f_1(\overrightarrow{x}), \cdots, f_n(\overrightarrow{x}))$ is a reduction from $(C,D)$ to $(A,B)$ is equivalent with:
\begin{enumerate}[(1)]
  \item $\overrightarrow{a}\in C\Rightarrow F(\overrightarrow{a})\in A$;
  \item $\overrightarrow{a}\in D\Rightarrow F(\overrightarrow{a})\in B$;
  \item $\overrightarrow{a}\notin C\cup D\Rightarrow F(\overrightarrow{a})\notin A\cup B$.
\end{enumerate}

\begin{theorem}\label{semi-DU is DU}
Let $(A,B)$  be  a disjoint pair  of $n$-ary RE relations. If $(A,B)$ is semi-$\sf DU$, then $(A,B)$ is $\sf DU$.
\end{theorem}

In the rest of this section, we prove Theorem \ref{semi-DU is DU}. Our proof uses Theorem \ref{coro of SDRT}. We first introduce the notions of 
$\sf EI$ and $\sf WEI$ theories. 

\begin{definition}\label{CEI}
Let $(A,B)$  be  a disjoint pair  of $n$-ary RE relations.
\begin{enumerate}[(1)]
  \item We say $(A,B)$ is \emph{$\sf EI$} if there is a recursive $n$-ary functional $F(x,y)$ on $\mathbb{N}^2$ such that for any $i,j\in\omega$, if $A\subseteq R^n_i, B\subseteq R^n_j$ and $R^n_i\cap R^n_j=\emptyset$, then $F(x,y)\notin R^n_i\cup R^n_j$.
  \item We say $(A,B)$ is \emph{$\sf WEI$} if there is a recursive $n$-ary functional $F(x,y)$ on $\mathbb{N}^2$ such that for any $i,j\in\omega$,
\begin{enumerate}[(i)]
  \item if $R^n_i=A$ and $R^n_j=B$, then $F(i,j)\notin A\cup B$;
  \item if $R^n_i=A$ and $R^n_j=B\cup \{F(i,j)\}$, then $F(i,j)\in A$;
  \item if $R^n_i=A\cup \{F(i,j)\}$ and $R^n_j=B$, then $F(i,j)\in B$.
\end{enumerate}
\end{enumerate}
\end{definition}

Our proof strategy of Theorem \ref{semi-DU is DU} is as follows.  For the definitions of $\sf CEI$ and $\sf KP$, we refer to Appendix \ref{1st proof}. In Appendix \ref{1st proof}, we prove that semi-$\sf DU\Rightarrow \sf KP\Rightarrow \sf CEI$ (see Theorem \ref{semi DU to KP} and Proposition \ref{KP to CEI}). 
Clearly, we have $\sf CEI\Rightarrow\sf EI\Rightarrow \sf WEI$. To prove Theorem \ref{semi-DU is DU}, it suffices to show that $\sf WEI$ implies $\sf DU$. In Theorem \ref{WEI is DU}, we prove that $\sf WEI$ implies $\sf DU$. To prove Theorem \ref{WEI is DU}, we first prove a lemma as follows, which uses a generalized version of Strong Double Recursion Theorem as in Theorem \ref{coro of SDRT}.

\begin{lemma}\label{key lemma for DU}
For any $n$-ary RE relations $A,B,C,D$ and any recursive $n$-ary functional $G(x,y)$ on $\mathbb{N}^2$, there exist $n$-ary recursive functions $f_1(\overrightarrow{y})$ and $f_2(\overrightarrow{y})$ such that for any $\overrightarrow{y}\in \mathbb{N}^n$, we have:
\begin{enumerate}[(1)]
  \item $\overrightarrow{y} \in B\Rightarrow R^n_{f_1(\overrightarrow{y})}=C\cup\{G(f_1(\overrightarrow{y}), f_2(\overrightarrow{y}))\}$;
  \item $\overrightarrow{y} \notin B\Rightarrow R^n_{f_1(\overrightarrow{y})}=C$;
  \item $\overrightarrow{y} \in A\Rightarrow R^n_{f_2(\overrightarrow{y})}=D\cup\{G(f_1(\overrightarrow{y}), f_2(\overrightarrow{y}))\}$;
  \item $\overrightarrow{y} \notin A\Rightarrow R^n_{f_2(\overrightarrow{y})}=D$.
\end{enumerate}
\end{lemma}
\begin{proof}\label{}
Define $3n$-ary RE relations $M_1(\overrightarrow{x},\overrightarrow{y}, \overrightarrow{z})\triangleq \overrightarrow{x}\in C \vee [\overrightarrow{x}=\overrightarrow{z}\wedge \overrightarrow{y}\in B]$ and $M_2(\overrightarrow{x},\overrightarrow{y}, \overrightarrow{z})\triangleq \overrightarrow{x}\in D \vee [\overrightarrow{x}=\overrightarrow{z}\wedge \overrightarrow{y}\in A]$.

Apply Theorem \ref{coro of SDRT} to $M_1(\overrightarrow{x},\overrightarrow{y}, \overrightarrow{z}), M_2(\overrightarrow{x},\overrightarrow{y}, \overrightarrow{z})$ and $G(x,y)$, there are recursive functions $f_1(\overrightarrow{y})$ and $f_2(\overrightarrow{y})$ such that:

\begin{align*}
\overrightarrow{x}\in R^n_{f_1(\overrightarrow{y})} \Leftrightarrow M_1(\overrightarrow{x}, \overrightarrow{y}, G(f_1(\overrightarrow{y}), f_2(\overrightarrow{y})));\\
\overrightarrow{x}\in R^n_{f_2(\overrightarrow{y})} \Leftrightarrow M_2(\overrightarrow{x}, \overrightarrow{y}, G(f_1(\overrightarrow{y}), f_2(\overrightarrow{y}))).
\end{align*}

Then we have:
\begin{enumerate}[(I)]
  \item $\overrightarrow{x}\in R^n_{f_1(\overrightarrow{y})} \Leftrightarrow  \overrightarrow{x}\in C \vee [\overrightarrow{x}=G(f_1(\overrightarrow{y}), f_2(\overrightarrow{y}))\wedge \overrightarrow{y}\in B]$;
  \item $\overrightarrow{x}\in R^n_{f_2(\overrightarrow{y})} \Leftrightarrow  \overrightarrow{x}\in D \vee [\overrightarrow{x}=G(f_1(\overrightarrow{y}), f_2(\overrightarrow{y}))\wedge \overrightarrow{y}\in A]$.
\end{enumerate}

From (I)-(II), we have  (1)-(4).
\end{proof}

\begin{theorem}\label{WEI is DU}
Let $(C,D)$  be  a disjoint pair  of $n$-ary RE relations. If $(C,D)$ is $\sf WEI$, then $(C,D)$ is $\sf DU$.
\end{theorem}
\begin{proof}\label{}
Suppose $(C,D)$  is $\sf WEI$ under a recursive $n$-ary functional 
\[G(x,y)=(g_1(x,y), \cdots, g_n(x,y))\] on $\mathbb{N}^2$.
Let $(A,B)$  be  any disjoint pair  of $n$-ary RE relations. We show that $(A,B)$  is reducible to $(C,D)$.

Apply Lemma \ref{key lemma for DU} to $A,B, C,D$ and $G(x,y)$. Then there exist recursive functions $f_1(\overrightarrow{y})$ and $f_2(\overrightarrow{y})$ such that (1)-(4) in Lemma \ref{key lemma for DU} hold.
Define a $n$-ary functional $H(\overrightarrow{y})\triangleq G(f_1(\overrightarrow{y}), f_2(\overrightarrow{y}))$ on $\mathbb{N}^n$.  Note that $H(\overrightarrow{y})$ is recursive.

\begin{claim}
The functional $H(\overrightarrow{y})$ is a reduction from $(A,B)$ to $(C,D)$.
\end{claim}
\begin{proof}\label{}
Suppose $\overrightarrow{y}\in A$. Then $R^n_{f_1(\overrightarrow{y})}=C$ and $R^n_{f_2(\overrightarrow{y})}=D\cup\{G(f_1(\overrightarrow{y}), f_2(\overrightarrow{y}))\}$. By Definition \ref{CEI}(ii),  $H(\overrightarrow{y})\in C$.

Suppose $\overrightarrow{y}\in B$. Then $R^n_{f_1(\overrightarrow{y})}=C\cup\{G(f_1(\overrightarrow{y}), f_2(\overrightarrow{y}))\}$ and $R^n_{f_2(\overrightarrow{y})}=D$. By Definition \ref{CEI}(iii),  $H(\overrightarrow{y})\in D$.

Suppose $\overrightarrow{y}\notin A\cup B$. Then $R^n_{f_1(\overrightarrow{y})}=C$ and $R^n_{f_2(\overrightarrow{y})}=D$. By Definition \ref{CEI}(i),  $H(\overrightarrow{y})\notin C\cup D$. Thus, $H(\overrightarrow{y})$ is a reduction from $(A,B)$ to $(C,D)$.
\end{proof}
Thus, $(C,D)$ is $\sf DU$.
\end{proof}

Clearly, we have $\sf CEI\Rightarrow\sf EI\Rightarrow \sf WEI$. Since semi-$\sf DU\Rightarrow \sf KP\Rightarrow \sf CEI$ from Appendix \ref{1st proof}, as a corollary of Theorem \ref{WEI is DU},  we have semi-$\sf DU$ implies $\sf DU$. This finishes the proof of Theorem \ref{semi-DU is DU}.

\subsection{Some applications in meta-mathematics of arithmetic}\label{apply of DU}

In this section, we first prove Theorem \ref{improve PST} which generalizes Theorem \ref{PST} to $n$-Rosser theories. Our proof of Theorem \ref{improve PST} uses Theorem \ref{semi-DU is DU}. Then we prove Theorem \ref{key improve thm} which essentially improves Theorem \ref{improve PST}. Our proof of Theorem \ref{key improve thm}  does not use Theorem \ref{semi-DU is DU}.

We first introduce the notion of ``strongly definable" for $n$-ary functionals.

\begin{definition}\label{}
Let $T$ be a consistent RE theory and $I: {\sf Num}\unlhd T$.
We say a $n$-ary functional $F(\overrightarrow{x})=(f_1(\overrightarrow{x}), \cdots, f_n(\overrightarrow{x}))$ on $\mathbb{N}^n$ is  \emph{strongly definable} in $T$ if for any $1\leq i\leq n$, there exists a formula $\varphi_i(\overrightarrow{x},y)$ of $(n+1)$-free variables such that for any $\overrightarrow{a}=(a_1, \cdots, a_n)\in \mathbb{N}^n, T\vdash \forall y [\varphi_i(\overline{a_1}^I, \cdots, \overline{a_n}^I,y)\leftrightarrow y=\overline{f_i(\overrightarrow{a})}^I].$
\end{definition}

Now we prove that ``$n$-Rosser" implies ``exact $n$-Rosser" under the assumption that any $n$-ary recursive functional on $\mathbb{N}^n$ is strongly definable in $T$. 

\begin{theorem}\label{improve PST}
Suppose $T$  is $n$-Rosser and any $n$-ary recursive functional on $\mathbb{N}^n$ is strongly definable in $T$, then $T$ is exact $n$-Rosser.
\end{theorem}
\begin{proof}\label{}
Suppose $T$  is $n$-Rosser via an interpretation $I: {\sf Num}\unlhd T$.
Take any $\sf DU$ pair of $n$-ary RE relations (e.g., $(U_1, U_2)$ in Proposition \ref{there exist DU}). Suppose $(U_1, U_2)$ is strongly separable by $\phi(x_1, \cdots, x_n)$ in $T$ w.r.t. $I$. Define $C=\{(a_1, \cdots, a_n)\in \mathbb{N}^n: T\vdash \phi(\overline{a_1}^I, \cdots, \overline{a_n}^I)\}$ and $D=\{(a_1, \cdots, a_n)\in\mathbb{N}^n: T\vdash \neg\phi(\overline{a_1}^I, \cdots, \overline{a_n}^I)\}$. Note that $U_1\subseteq C$ and $U_2\subseteq D$. Since $(U_1, U_2)$ is $\sf DU$, $(C,D)$ is semi$-\sf DU$. By Theorem \ref{semi-DU is DU}, $(C,D)$ is $\sf DU$. Note that $(C,D)$ is exactly separable by $\phi(x_1, \cdots, x_n)$ in $T$  w.r.t. $I$.

Let $(A,B)$  be  a disjoint pair  of $n$-ary RE relations. Since $(C,D)$ is $\sf DU$, let $F(\overrightarrow{x})=(f_1(\overrightarrow{x}), \cdots, f_n(\overrightarrow{x}))$ be a  recursive $n$-ary functional on $\mathbb{N}^n$ that reduces $(A,B)$  to $(C,D)$.

Suppose that for any $1\leq i\leq n$, there exists a formula $\psi_i(x_1, \cdots, x_n,y)$ such that  for any $\overrightarrow{a}=(a_1, \cdots, a_n)\in \mathbb{N}^n, T\vdash \forall y [\psi_i(\overline{a_1}^I, \cdots, \overline{a_n}^I,y)\leftrightarrow y=\overline{f_i(\overrightarrow{a})}^I]$.

Given $\overrightarrow{x}=(x_1, \cdots, x_n)$, define $\theta(\overrightarrow{x})\triangleq\exists y_1\cdots\exists y_n [\psi_1(\overrightarrow{x},y_1)\wedge\cdots\wedge \psi_n(\overrightarrow{x},y_n)\wedge \phi(y_1, \cdots, y_n)]$.
Note that for any $\overrightarrow{a}=(a_1, \cdots, a_n)\in \mathbb{N}^n$, $T\vdash  \theta(\overline{a_1}^I, \cdots, \overline{a_n}^I)\leftrightarrow \phi(\overline{f_1(\overrightarrow{a})}^I,\cdots, \overline{f_n(\overrightarrow{a})}^I)$.

\begin{claim}
$\theta(\overrightarrow{x})$ exactly separates $A$ from $B$ in $T$ w.r.t. $I$.
\end{claim}
\begin{proof}\label{}
For any $\overrightarrow{a}=(a_1, \cdots, a_n)\in \mathbb{N}^n$, we have:
\[\overrightarrow{a}\in A \Leftrightarrow F(\overrightarrow{a})\in C \Leftrightarrow T\vdash \phi(\overline{f_1(\overrightarrow{a})}^I,\cdots, \overline{f_n(\overrightarrow{a})}^I) \Leftrightarrow T\vdash  \theta(\overline{a_1}^I, \cdots, \overline{a_n}^I);\]
\[\overrightarrow{a}\in B \Leftrightarrow F(\overrightarrow{a})\in D \Leftrightarrow T\vdash \neg\phi(\overline{f_1(\overrightarrow{a})}^I,\cdots, \overline{f_n(\overrightarrow{a})}^I) \Leftrightarrow T\vdash   \neg\theta(\overline{a_1}^I, \cdots, \overline{a_n}^I).\]
\end{proof}
Thus, $T$ is exact $n$-Rosser.
\end{proof}

Now we introduce the notion of admissible $n$-ary functionals. We will improve Theorem \ref{improve PST} via this notion. 

\begin{definition}\label{}
Let $T$ be a consistent RE theory and $I: {\sf Num}\unlhd T$.
We say a $n$-ary functional $F(\overrightarrow{x})=(f_1(\overrightarrow{x}), \cdots, f_n(\overrightarrow{x}))$ on $\mathbb{N}^m$ is \emph{admissible} in $T$  if for any formula $\phi(x_1, \cdots, x_n)$,  there exists a formula $\psi(x_1, \cdots, x_m)$ such that for any $\overrightarrow{a}=(a_1, \cdots, a_m)\in \mathbb{N}^m$, we have $T\vdash  \psi(\overline{a_1}^I, \cdots, \overline{a_m}^I)\leftrightarrow \phi(\overline{f_1(\overrightarrow{a})}^I,\cdots, \overline{f_n(\overrightarrow{a})}^I)$.
\end{definition}

It is easy to check that if a $n$-ary functional $F(\overrightarrow{x})=(f_1(\overrightarrow{x}), \cdots, f_n(\overrightarrow{x}))$ on $\mathbb{N}^n$ is strongly definable in $T$, then $F(\overrightarrow{x})$ is admissible in $T$.

\begin{corollary}
Let $T$ be a consistent RE theory. Suppose $T$  is $n$-Rosser  and any $n$-ary recursive functional on $\mathbb{N}^n$ is admissible in $T$, then $T$ is exact $n$-Rosser.
\end{corollary}
\begin{proof}\label{}
From the proof of Theorem \ref{improve PST},  for any $\overrightarrow{a}=(a_1, \cdots, a_n)\in \mathbb{N}^n$, we have:
\[\overrightarrow{a}\in A \Leftrightarrow  T\vdash \phi(\overline{f_1(\overrightarrow{a})}^I,\cdots, \overline{f_n(\overrightarrow{a})}^I);\]
\[\overrightarrow{a}\in B \Leftrightarrow T\vdash \neg\phi(\overline{f_1(\overrightarrow{a})}^I,\cdots, \overline{f_n(\overrightarrow{a})}^I).\]

Since the $n$-ary recursive functional $F(\overrightarrow{x})=(f_1(\overrightarrow{x}), \cdots, f_n(\overrightarrow{x}))$ on $\mathbb{N}^n$ is admissible in $T$, then there exists a formula $\psi(x_1, \cdots, x_n)$ such that for any $\overrightarrow{a}=(a_1, \cdots, a_n)\in \mathbb{N}^n$, we have $T\vdash  \psi(\overline{a_1}^I, \cdots, \overline{a_n}^I)\leftrightarrow \phi(\overline{f_1(\overrightarrow{a})}^I,\cdots, \overline{f_n(\overrightarrow{a})}^I)$. Thus, $\psi(x_1, \cdots, x_n)$ exactly separates $A$ from $B$ in $T$ w.r.t. $I$.
\end{proof}

Finally, we improve Theorem \ref{improve PST}. We will show that for Theorem \ref{improve PST}, it suffices to assume that for any $h$, the $n$-ary functional $F(\overrightarrow{x})=\overbrace{J_{n+1}(h, \overrightarrow{x})}$ on $\mathbb{N}^n$ is admissible.
To prove Theorem \ref{key improve thm}, we first introduce a notion and prove a lemma as follows. 

\begin{definition}\label{}
Given a consistent RE theory $T$ with $I: {\sf Num}\unlhd T$, a formula $\varphi(v_1, \cdots, v_n)$ and $\overbrace{x}=(x, \cdots, x)\in \mathbb{N}^n$, we define $\varphi(\overbrace{x})=\varphi(\overline{x}^I, \cdots, \overline{x}^I)$.
\end{definition}

%\item We say $J_{n+1}(h, \overrightarrow{x})$ is admissible in $S$ if for any formula $\varphi(x_1, \cdots, x_n)$, there exists a formula $\psi(x_1, \cdots, x_n)$ such that for any $\overrightarrow{a}=(a_1, \cdots, a_n)\in N^n$, we have $S\vdash \psi(\overline{a_1}, \cdots, \overline{a_n})\leftrightarrow \varphi(\overbrace{J_{n+1}(h, \overrightarrow{a})})$.

\begin{lemma}\label{lemma for exact Rosser}
Suppose $T$ is $n$-Rosser. For any disjoint $(n+1)$-ary relations $M_1(\overrightarrow{x},y)$ and $M_2(\overrightarrow{x},y)$, there is a formula $\varphi(\overrightarrow{x})$ of n-free variables with code $h$ such that for any $\overrightarrow{a}\in \mathbb{N}^n$, $M_1(\overrightarrow{a},h)\Leftrightarrow T\vdash \varphi(\overbrace{J_{n+1}(h, \overrightarrow{a})})$ and $M_2(\overrightarrow{a},h)\Leftrightarrow T\vdash \neg\varphi(\overbrace{J_{n+1}(h, \overrightarrow{a})})$.
\end{lemma}
\begin{proof}\label{}
For any $x\in \mathbb{N}$, define $\overbrace{x}=(x, \cdots, x)\in \mathbb{N}^n$.
Define
\[A=\{\overbrace{J_{n+1}(y, \overrightarrow{x})}: M_1(\overrightarrow{x},y)\vee T\vdash \neg\varphi(\overbrace{J_{n+1}(y, \overrightarrow{x})})\}\] and
\[B=\{\overbrace{J_{n+1}(y, \overrightarrow{x})}: M_2(\overrightarrow{x},y)\vee T\vdash \varphi(\overbrace{J_{n+1}(y, \overrightarrow{x})})\}.\]
Since $T$ is $n$-Rosser, there is a formula $\varphi(\overrightarrow{x})$ of $n$-free variables with code $h$ such that $\varphi(\overrightarrow{x})$ strongly separates $A-B$ from $B-A$ in $T$ w.r.t. $I$.  Then for any $\overrightarrow{a}\in \mathbb{N}^n$,
\[[M_1(\overrightarrow{a},h)\vee T\vdash \neg\varphi(\overbrace{J_{n+1}(h, \overrightarrow{a})})]\wedge \neg [M_2(\overrightarrow{a},h)\vee T\vdash \varphi(\overbrace{J_{n+1}(h, \overrightarrow{a})})] \Rightarrow T\vdash \varphi(\overbrace{J_{n+1}(h, \overrightarrow{a})})\] and 
\[[M_2(\overrightarrow{a},h)\vee T\vdash \varphi(\overbrace{J_{n+1}(h, \overrightarrow{a})})]\wedge \neg [M_1(\overrightarrow{a},h)\vee T\vdash \neg\varphi(\overbrace{J_{n+1}(h, \overrightarrow{a})})] \Rightarrow T\vdash \neg\varphi(\overbrace{J_{n+1}(h, \overrightarrow{a})}).\]
Then we have: \[M_1(\overrightarrow{a},h)\Leftrightarrow T\vdash \varphi(\overbrace{J_{n+1}(h, \overrightarrow{a})})\] and \[M_2(\overrightarrow{a},h)\Leftrightarrow T\vdash \neg\varphi(\overbrace{J_{n+1}(h, \overrightarrow{a})}).\]
\end{proof}

\begin{theorem}\label{key improve thm}
If $T$ is $n$-Rosser and for any $h$, the $n$-ary functional $F(\overrightarrow{x})=\overbrace{J_{n+1}(h, \overrightarrow{x})}$ on $\mathbb{N}^n$ is admissible in $T$, then $T$ is exact $n$-Rosser.
\end{theorem}
\begin{proof}\label{}
Suppose $S_1(\overrightarrow{x})$ and $S_2(\overrightarrow{x})$ are disjoint $n$-ary RE relations. Define $M_1(\overrightarrow{x},y)\triangleq  S_1(\overrightarrow{x})$  and $M_2(\overrightarrow{x},y)\triangleq S_2(\overrightarrow{x})$.

Apply Lemma \ref{lemma for exact Rosser} to $M_1(\overrightarrow{x},y)$ and $M_2(\overrightarrow{x},y)$. Then there is a formula $\varphi(\overrightarrow{x})$ of $n$-free variables with code $h$ such that for any $\overrightarrow{a}=(a_1, \cdots, a_n)\in \mathbb{N}^n, S_1(\overrightarrow{a})\Leftrightarrow T\vdash \varphi(\overbrace{J_{n+1}(h, \overrightarrow{a})})$ and $S_2(\overrightarrow{a})\Leftrightarrow T\vdash \neg\varphi(\overbrace{J_{n+1}(h, \overrightarrow{a})})$.
Since $F(\overrightarrow{x})$ is admissible in $T$,  there exists a formula $\psi(\overrightarrow{x})$ of $n$-free variables  such that $T\vdash \psi(\overline{a_1}^I, \cdots, \overline{a_n}^I)\leftrightarrow  \varphi(\overbrace{J_{n+1}(h, \overrightarrow{a})})$.
Thus, $\psi(\overrightarrow{x})$ exactly separates $S_1(\overrightarrow{x})$ from $S_2(\overrightarrow{x})$ in $T$ w.r.t. $I$. Hence, $T$ is exact $n$-Rosser.
\end{proof}

\section{Equivalences under the definability of the paring function}\label{equi under def}

In this section, we will show that, assuming that the pairing function $J_2(x,y)$ is strongly definable in the base theory, then for any $n\geq 1$, we have:
\begin{enumerate}[(1)]
  \item $n$-Rosser implies   $(n+1)$-Rosser;
  \item   exact $n$-Rosser implies  exact $(n+1)$-Rosser;
  \item  effectively $n$-Rosser implies  effectively $(n+1)$-Rosser;
  \item  effectively exact $n$-Rosser implies  effectively exact $(n+1)$-Rosser.
\end{enumerate}

\begin{theorem}\label{}
Assume that the pairing function $J_2(x,y)$ is strongly definable in $T$. Then for any $n\geq 1$, if $T$ is $n$-Rosser, then $T$ is $(n+1)$-Rosser.
\end{theorem}
\begin{proof}\label{}
Suppose $T$ is $n$-Rosser with $I: {\sf Num}\unlhd T$. We show $T$ is $(n+1)$-Rosser.
Let $M^{n+1}_1$ and $M^{n+1}_2$ be two $(n+1)$-ary RE relations. It suffices to find a formula $\phi(x_1,\cdots, x_{n+1})$ such that $\phi(x_1,\cdots, x_{n+1})$ strongly separates $M^{n+1}_1-M^{n+1}_2$ from $M^{n+1}_2-M^{n+1}_1$ in $T$ w.r.t. $I$.

Define $n$-ary relations $Q_1$ and $Q_2$ as follows where $K$ and $L$ are recursive functions with the property that $K(J(a,b))=a$ and $L(J(a,b))=b$:
\begin{align*}
Q_1(a_1, \cdots, a_n)\Leftrightarrow M^{n+1}_1(a_1, \cdots, a_{n-1}, Ka_n, La_n);\\
Q_2(a_1, \cdots, a_n)\Leftrightarrow M^{n+1}_2(a_1, \cdots, a_{n-1}, Ka_n, La_n).
\end{align*}

Let $A(x_1, \cdots, x_n)$ be the formula that strongly separates $Q_1-Q_2$ from $Q_2-Q_1$ in $T$ w.r.t. $I$. Let $\theta(x,y,z)$ be the formula that strongly defines the pairing function $J_2(x,y)$ in $T$. Define $\phi(x_1,\cdots, x_{n+1})\triangleq\exists z(\theta(x_n, x_{n+1},z)\wedge A(x_1, \cdots, x_{n-1}, z))$.

\begin{claim}
The formula $\phi(x_1,\cdots, x_{n+1})$ strongly separates $M^{n+1}_1-M^{n+1}_2$ from $M^{n+1}_2-M^{n+1}_1$ in $T$ w.r.t. $I$.
\end{claim}
\begin{proof}\label{}
Suppose $(a_1, \cdots, a_{n+1})\in M^{n+1}_1-M^{n+1}_2$. Then $(a_1, \cdots, a_{n-1}, J(a_n, a_{n+1}))\in Q_1-Q_2$. Thus, $T\vdash A(\overline{a_1}^I, \cdots, \overline{a_{n-1}}^I, \overline{J(a_n, a_{n+1})}^I)$.
Note that $\phi(\overline{a_1}^I,\cdots, \overline{a_{n+1}}^I)$ is just $\exists z(\theta(\overline{a_n}^I,\overline{a_{n+1}}^I,z)\wedge A(\overline{a_1}^I, \cdots, \overline{a_{n-1}}^I,z))$. Thus, we have  $T\vdash \phi(\overline{a_1}^I,\cdots, \overline{a_{n+1}}^I)$ since $T\vdash \phi(\overline{a_1}^I,\cdots, \overline{a_{n+1}}^I)\leftrightarrow A(\overline{a_1}^I, \cdots, \overline{a_{n-1}}^I, \overline{J(a_n, a_{n+1})}^I)$.

Similarly, if $(a_1, \cdots, a_{n+1})\in M^{n+1}_2-M^{n+1}_1$, then $T\vdash \neg\phi(\overline{a_1}^I,\cdots, \overline{a_{n+1}}^I)$.
\end{proof}

Thus, $T$ is $(n+1)$-Rosser w.r.t. $I$.
\end{proof}

By a similar argument, we can show that:
\begin{theorem}\label{}
Assume that the pairing function $J_2(x,y)$ is strongly definable in $T$. Then for any $n\geq 1$, if $T$ is exact $n$-Rosser, then $T$ is exact $(n+1)$-Rosser.
\end{theorem}

\begin{theorem}\label{}
Assume that the  pairing function $J_2(x,y)$ is strongly definable in $T$. Then if $T$ is effectively $n$-Rossoer, then $T$ is effectively $(n+1)$-Rosser.
\end{theorem}
\begin{proof}\label{}

Let $f$ be the witness function for $T$ being effectively $n$-Rosser with $I: {\sf Num}\unlhd T$, i.e., for any $i,j\in\omega$, $f(i,j)$ codes a formula with $n$ free variables which strong separates $R^n_i-R^n_j$ from $R^n_j-R^n_i$ in $T$ w.r.t. $I$.

Let $\theta(x,y,z)$ be the formula which strongly defines the pairing function $J_2(x,y)$ in $T$.
Take a recursive function $g$ such that it maps the code of $\phi(x_1, \cdots, x_n)$ to the code of $\exists y[\phi(x_1, \cdots, x_{n-1}, y)\wedge \theta(x_{n}, x_{n+1},y)]$.

Note that by s-m-n theorem, there exists a recursive function $t(x)$ such that for any $i\in\omega$, \[(a_1, \cdots, a_{n})\in R^{n}_{t(i)}\Leftrightarrow (a_1, \cdots, a_{n-1}, Ka_{n}, La_{n})\in R^{n+1}_i.\]

Define $h(i,j)=g(f(t(i), t(j)))$. Note that $h$ is recursive. We show that for any $(n+1)$-ary RE relations $R^{n+1}_i$ and $R^{n+1}_j$, $h(i,j)$ codes a formula with $(n+1)$-free variables that strongly separates $R^{n+1}_i-R^{n+1}_j$ from $R^{n+1}_j-R^{n+1}_i$ in $T$ w.r.t. $I$.

Note that $f(t(i), t(j))$ codes a formula with $n$-free variables which strongly separates $R^{n}_{t(i)}-R^{n}_{t(j)}$ from $R^{n}_{t(j)}-R^{n}_{t(i)}$ in $T$ w.r.t. $I$. Let $f(t(i), t(j))\triangleq \ulcorner \phi(x_1, \cdots, x_n)\urcorner$. Then $h(i,j)=\ulcorner \psi(x_1, \cdots, x_{n+1})  \urcorner$ where $\psi(x_1, \cdots, x_{n+1})=\exists y[\phi(x_1, \cdots, x_{n-1}, y)\wedge \theta(x_{n}, x_{n+1},y)]$.

\begin{claim}
The formula $\psi(x_1, \cdots, x_{n+1})$ strongly separates $R^{n+1}_i-R^{n+1}_j$ from $R^{n+1}_j-R^{n+1}_i$ in $T$ w.r.t. $I$.
\end{claim}
\begin{proof}\label{}

Suppose $(a_1, \cdots, a_{n+1})\in R^{n+1}_i-R^{n+1}_j$. Then $(a_1, \cdots, a_{n-1}, J(a_n, a_{n+1}))\in R^{n}_{t(i)}-R^{n}_{t(j)}$.
 Then $T\vdash \phi(\overline{a_1}^I, \cdots, \overline{a_{n-1}}^I, \overline{J(a_n, a_{n+1})}^I)$.

Note that $T\vdash \exists y[\phi(\overline{a_1}^I, \cdots, \overline{a_{n-1}}^I, y)\wedge \theta(\overline{a_{n}}^I, \overline{a_{n+1}}^I,y)]\leftrightarrow  \phi(\overline{a_1}^I, \cdots, \overline{a_{n-1}}^I, \overline{J(a_n, a_{n+1})}^I)$.
Thus, $T\vdash \psi(\overline{a_1}^I, \cdots, \overline{a_{n+1}}^I)$.

Similarly, if $(a_1, \cdots, a_{n+1})\in R^{n+1}_j-R^{n+1}_i$, then $T\vdash \neg \psi(\overline{a_1}^I, \cdots, \overline{a_{n+1}}^I)$.
\end{proof}

Thus, $T$ is effectively $(n+1)$-Rosser under  $h$ and $I$.

\end{proof}

As a corollary of Theorem \ref{effective versus effective exact},  assuming the pairing function $J_2(x,y)$ is strongly definable in $T$, if $T$ is effectively exact $n$-Rossoer, then $T$ is effectively exact $(n+1)$-Rosser.

%It is a long open question whether 1-Rosser implies exact 1-Rosser. We show that if the pairing function $J_2(x,y)$ is strongly definable, then they are equivalent.

\begin{corollary}\label{rela among four hier}
If the pairing function $J_2(x,y)$ is strongly definable in $T$, then for any $n\geq 1$, the following are equivalent:
\begin{enumerate}[(1)]
  \item $T$ is $n$-Rosser;
  \item $T$ is effectively $n$-Rosser;
  \item  $T$ is exact $n$-Rosser;
  \item $T$ is effectively exact $n$-Rosser.
\end{enumerate}
\end{corollary}
\begin{proof}\label{}
Follows from the following facts: $n$-Rosser $\Rightarrow (n+1)$-Rosser $\Rightarrow$ effectively $n$-Rosser; $n$-Rosser $\Rightarrow (n+1)$-Rosser $\Rightarrow$ exact $n$-Rosser; and effectively $n$-Rosser $\Leftrightarrow$ effectively exact $n$-Rosser.
\end{proof}

In summary, we have: 
\begin{theorem}\label{}
If the paring function $J_2(x,y)$ is strongly definable in $T$, then the following are
equivalent:
\begin{enumerate}[(1)]
  \item $T$ is Rosser;
  \item $T$ is $n$-Rosser for some $n\geq 1$;
  \item $T$ is effectively $n$-Rosser for some $n\geq 1$;
  \item $T$ is exact $n$-Rosser for some $n\geq 1$;
  \item $T$ is effectively exact $n$-Rosser for some $n\geq 1$.
\end{enumerate}
\end{theorem}

The study of the generalized hierarchy of $n$-Rosser theories, exact $n$-Rosser theories, effectively $n$-Rosser theories and effectively exact $n$-Rosser theories, which has been pursued in this paper, also leads to some new insights in the understanding of formal systems. Let us take two examples. Firstly, it is well known that the theory $\mathbf{R}$ is Rosser for RE sets in the literature. In this paper, we have shown that the theory $\mathbf{R}$ is effectively exact $n$-Rosser for any $n\geq 1$, which tells us more information about the theory $\mathbf{R}$. Secondly, at first sight, the notion of effectively exact $n$-Rosser is stronger than the notion of $n$-Rosser. By Theorem \ref{rela among four hier}, if the pairing function $J_2(x,y)$ is strongly definable in $T$, then ``$T$ is $n$-Rosser" is equivalent with ``$T$ is effectively exact $n$-Rosser". If  the pairing function $J_2(x,y)$ is not strongly definable in a theory, this theory must be very weak. For all natural mathematical theories we know, the pairing function $J_2(x,y)$ is  strongly definable in them. Thus, for natural mathematical theories, there is no difference between the notion of $n$-Rosser and the notion of effectively exact $n$-Rosser.

We conclude the paper with a question.

\begin{question}\label{}
Does $1$-Rosser imply exact $1$-Rosser? Generally, does $n$-Rosser imply exact $n$-Rosser?
\end{question}

%\begin{enumerate}[(A)]
  %\item Could we find a 1-Rosser theory $T$ such that $T$ is strictly weaker than $\mathbf{R}$?
  %\item Is there a minimal n-Rosser theory w.r.t. interpretation?
%\end{enumerate}

Since $2$-Rosser implies exact $1$-Rosser and effectively $1$-Rosser is equivalent with effectively exact $1$-Rosser,  if $1$-Rosser does not imply exact $1$-Rosser, then $1$-Rosser does not imply $2$-Rosser, and $1$-Rosser does not imply effectively $1$-Rosser.

\appendix

\section{The second proof of Theorem \ref{semi-DU is DU}}\label{1st proof}

In this Appendix, we first give a proof of Theorem \ref{SDRT}, then we give a second proof of Theorem \ref{semi-DU is DU}.

We first give a proof of Theorem \ref{SDRT} as follows. Given two  $(n+2m+2)$-ary RE relations 
 $M_1(\overrightarrow{x}, \overrightarrow{y_1}, \overrightarrow{y_2}, z_1,z_2)$ and  $M_2(\overrightarrow{x}, \overrightarrow{y_1}, \overrightarrow{y_2}, z_1,z_2)$, we show that there are $2m$-ary recursive functions $t_1(\overrightarrow{y_1}, \overrightarrow{y_2})$ and $t_2(\overrightarrow{y_1}, \overrightarrow{y_2})$ such that for any $\overrightarrow{y_1}, \overrightarrow{y_2}\in \mathbb{N}^m$,
\begin{enumerate}[(1)]
  \item $\overrightarrow{x}\in R^n_{t_1(\overrightarrow{y_1}, \overrightarrow{y_2})}\Leftrightarrow M_1(\overrightarrow{x}, \overrightarrow{y_1}, \overrightarrow{y_2}, t_1(\overrightarrow{y_1}, \overrightarrow{y_2}), t_2(\overrightarrow{y_1}, \overrightarrow{y_2}))$;
  \item $\overrightarrow{x}\in R^n_{t_2(\overrightarrow{y_1}, \overrightarrow{y_2})}\Leftrightarrow M_2(\overrightarrow{x}, \overrightarrow{y_1}, \overrightarrow{y_2}, t_1(\overrightarrow{y_1}, \overrightarrow{y_2}), t_2(\overrightarrow{y_1}, \overrightarrow{y_2}))$.
\end{enumerate}

Let $a$ be an index of $M_1(\overrightarrow{x}, \overrightarrow{y_1}, \overrightarrow{y_2}, z_1,z_2)$ and $b$ be an index of $M_2(\overrightarrow{x}, \overrightarrow{y_1}, \overrightarrow{y_2}, z_1,z_2)$.

\begin{claim}
There is a $(2m+3)$-ary recursive function $f(z,z_1,z_2, \overrightarrow{y_1}, \overrightarrow{y_2})$ such that for any $z, z_1,z_2\in\omega$ and $\overrightarrow{y_1}, \overrightarrow{y_2}\in \mathbb{N}^m$, we have:
\[\overrightarrow{x}\in R^n_{f(z,z_1,z_2, \overrightarrow{y_1}, \overrightarrow{y_2})}\Leftrightarrow R^{n+2m+2}_z(\overrightarrow{x}, \overrightarrow{y_1}, \overrightarrow{y_2}, f(z_1,z_1,z_2, \overrightarrow{y_1}, \overrightarrow{y_2}), f(z_2,z_1,z_2, \overrightarrow{y_1}, \overrightarrow{y_2})).\]
\end{claim}
\begin{proof}\label{}
By s-m-n theorem, there exists a $(2m+4)$-ary recursive function $g(z,z_1,z_2, \overrightarrow{y_1}, \overrightarrow{y_2},s)$ such that $\overrightarrow{x}\in R^n_{g(z,z_1,z_2, \overrightarrow{y_1}, \overrightarrow{y_2},s)}$ iff  $R^{n+2m+4}_s(\overrightarrow{x}, z,z_1,z_2, \overrightarrow{y_1}, \overrightarrow{y_2},s)$. Let $h$ be an index of the following relation on $(\overrightarrow{x}, z,z_1,z_2, \overrightarrow{y_1}, \overrightarrow{y_2},s)$:
\[R^{n+2m+2}_z(\overrightarrow{x}, \overrightarrow{y_1}, \overrightarrow{y_2}, g(z_1,z_1,z_2, \overrightarrow{y_1}, \overrightarrow{y_2}, s), g(z_2,z_1,z_2, \overrightarrow{y_1}, \overrightarrow{y_2}, s)).\]

Define $f(z,z_1,z_2, \overrightarrow{y_1}, \overrightarrow{y_2})\triangleq g(z,z_1,z_2, \overrightarrow{y_1}, \overrightarrow{y_2},h)$. Then:
\begin{align*}
 \overrightarrow{x}\in R^n_{f(z,z_1,z_2, \overrightarrow{y_1}, \overrightarrow{y_2})}  &\Leftrightarrow R^{n+2m+4}_h(\overrightarrow{x}, z,z_1,z_2, \overrightarrow{y_1}, \overrightarrow{y_2},h) \\
      &\Leftrightarrow R^{n+2m+2}_z(\overrightarrow{x}, \overrightarrow{y_1}, \overrightarrow{y_2}, f(z_1,z_1,z_2, \overrightarrow{y_1}, \overrightarrow{y_2}), f(z_2,z_1,z_2, \overrightarrow{y_1}, \overrightarrow{y_2})).
\end{align*}
\end{proof}

Define $2m$-ary functions $t_1(\overrightarrow{y_1}, \overrightarrow{y_2})\triangleq f(a,a,b, \overrightarrow{y_1}, \overrightarrow{y_2})$ and $t_2(\overrightarrow{y_1}, \overrightarrow{y_2})\triangleq f(b,a,b, \overrightarrow{y_1}, \overrightarrow{y_2})$. Then we have:
\begin{align*}
 \overrightarrow{x}\in R^n_{t_1(\overrightarrow{y_1}, \overrightarrow{y_2})} &\Leftrightarrow R^{n+2m+2}_a(\overrightarrow{x}, \overrightarrow{y_1}, \overrightarrow{y_2}, t_1(\overrightarrow{y_1}, \overrightarrow{y_2}), t_2(\overrightarrow{y_1}, \overrightarrow{y_2}))\\
      &\Leftrightarrow M_1(\overrightarrow{x}, \overrightarrow{y_1}, \overrightarrow{y_2}, t_1(\overrightarrow{y_1}, \overrightarrow{y_2}), t_2(\overrightarrow{y_1}, \overrightarrow{y_2}));\\
 \overrightarrow{x}\in R^n_{t_2(\overrightarrow{y_1}, \overrightarrow{y_2})} &\Leftrightarrow R^{n+2m+2}_b(\overrightarrow{x}, \overrightarrow{y_1}, \overrightarrow{y_2}, t_1(\overrightarrow{y_1}, \overrightarrow{y_2}), t_2(\overrightarrow{y_1}, \overrightarrow{y_2}))\\
      &\Leftrightarrow M_2(\overrightarrow{x}, \overrightarrow{y_1}, \overrightarrow{y_2}, t_1(\overrightarrow{y_1}, \overrightarrow{y_2}), t_2(\overrightarrow{y_1}, \overrightarrow{y_2})).
\end{align*}
This finish the proof of Theorem \ref{SDRT}.

Now, we give a second proof of Theorem \ref{semi-DU is DU}.
The notions of $\sf KP, \sf CEI$ and $\sf DG$ are introduced in \cite{Smullyan93} for RE sets.  In this section, we first introduce the notions of $\sf KP, \sf CEI$ and $\sf DG$ for a disjoint pair  of $n$-ary RE relations via the notion of $n$-ary functionals. Then we prove that for any disjoint pair  of $n$-ary RE relations with $n\geq 1$,  Semi-$\sf DU\Rightarrow \sf KP\Rightarrow \sf CEI\Rightarrow \sf DG\Rightarrow DU$. As a corollary, Semi-$\sf DU$ implies $\sf DU$.

We first introduce the notion of \emph{$\sf KP$} for a disjoint pair  of $n$-ary RE relations.
\begin{definition}\label{}
Let $(A,B)$  be  a disjoint pair  of $n$-ary RE relations. We say $(A,B)$ is \emph{$\sf KP$} if there exists a recursive $n$-ary  functional $F(x,y)=(f_1(x,y), \cdots, f_n(x,y))$ on $\mathbb{N}^2$ such that for any $x, y\in\omega$,
\begin{enumerate}[(i)]
  \item  $F(x,y)\in R^n_y-R^n_x \Rightarrow F(x,y)\in A$;
  \item $F(x,y)\in R^n_x-R^n_y \Rightarrow F(x,y)\in B$.
\end{enumerate}
\end{definition}

Now we construct a disjoint pair  of $n$-ary RE relations which is $\sf KP$.

\begin{theorem}\label{comparison lemma}
There is an $(n+2)$-ary RE relation $B(\overrightarrow{x},y,z)$ (which we read $\overrightarrow{x}\in R^n_y$ before $\overrightarrow{x}\in R^n_z$) such that for any $i,j\in\omega$, we have:
\begin{enumerate}[(1)]
  \item $\{\overrightarrow{x}: B(\overrightarrow{x},i,j)\}\cap \{\overrightarrow{x}: B(\overrightarrow{x},j,i)\}=\emptyset$.
  \item $R^n_i-R^n_j\subseteq\{\overrightarrow{x}: B(\overrightarrow{x},i,j)\}$ and $R^n_j-R^n_i\subseteq\{\overrightarrow{x}: B(\overrightarrow{x},j,i)\}$.
  \item If $R^n_i$ and $R^n_j$ are disjoint, then $R^n_i=\{\overrightarrow{x}: B(\overrightarrow{x},i,j)\}$ and $R^n_j=\{\overrightarrow{x}: B(\overrightarrow{x},j,i)\}$.
\end{enumerate}
\end{theorem}
\begin{proof}\label{}
Since the relation $\overrightarrow{x}\in R^n_y$ is RE, there is a recursive $(n+2)$-ary $\Delta^0_0$ relation $P(\overrightarrow{x}, y,z)$ such that $\overrightarrow{x}\in R^n_y\Leftrightarrow \exists z P(\overrightarrow{x}, y,z)$.

Define $B(\overrightarrow{x},y,z)\triangleq \exists s [P(\overrightarrow{x}, y,s)\wedge \forall t\leq s\neg P(\overrightarrow{x}, z,t)]$ which says that $\overrightarrow{x}\in R^n_y$ before $\overrightarrow{x}\in R^n_z$. Note that $B(\overrightarrow{x},y,z)$  is an $(n+2)$-ary RE relation. It is easy to check that properties (1)-(3) hold.
\end{proof}

\begin{proposition}\label{eg of KP}
There exists a   pair  of $n$-ary RE relations which is $\sf KP$.
\end{proposition}
\begin{proof}\label{}
Recall that for any $x \in \mathbb{N}$, $\overbrace{x}$ denotes $(x, \cdots, x)\in \mathbb{N}^n$.  Recall the $(n+2)$-ary RE relation $B(\overrightarrow{x}, y,z)$ in Theorem \ref{comparison lemma}. Define:
\[K_1=\{\overbrace{x}\in \mathbb{N}^n: B(\overbrace{x}, Lx, Kx)\}; K_2=\{\overbrace{x}\in \mathbb{N}^n: B(\overbrace{x}, Kx, Lx)\}.\]
Define $F(x,y)=\overbrace{J_2(x,y)}$. Note that $F(x,y)$ is a recursive $n$-ary functional on $\mathbb{N}^2$. We show that $(K_1,K_2)$ is $\sf KP$ under $F(x,y)$.
For any $x, y\in\omega$, by Theorem \ref{comparison lemma}, $F(x,y)\in R^n_{y}-R^n_{x}\Rightarrow \overbrace{J_2(x,y)}\in K_1\Rightarrow F(x,y)\in K_1$. Similarly, we have $F(x,y)\in R^n_{x}-R^n_{y}\Rightarrow F(x,y)\in K_2$.
\end{proof}

Now we show that semi-$\sf DU$ implies $\sf KP$.

\begin{theorem}\label{semi DU to KP}
Let $(A,B)$  be  a disjoint pair  of $n$-ary RE relations. If $(A,B)$ is semi-$\sf DU$, then $(A,B)$ is $\sf KP$.
\end{theorem}
\begin{proof}\label{}
Let $(C,D)$  be  a disjoint pair  of $n$-ary RE relations. By Proposition \ref{eg of KP}, it suffices to show  if  $(C,D)$  is $\sf KP$ and $(C,D)$ is semi-reducible to $(A,B)$, then $(A,B)$ is $\sf KP$.

Suppose $(C,D)$  is $\sf KP$ under a recursive $n$-ary functional $F(x,y)=(f_1(x,y), \cdots, f_n(x,y))$ on $\mathbb{N}^2$, and $G(\overrightarrow{x})=(g_1(\overrightarrow{x}), \cdots, g_n(\overrightarrow{x}))$ is a $n$-ary recursive functional on $\mathbb{N}^n$ and $G(\overrightarrow{x})$ is a semi-reduction from $(C,D)$ to $(A,B)$. By s-m-n theorem, there exists a recursive function $t(y)$ such that for any $\overrightarrow{x}\in \mathbb{N}^n$, \[\overrightarrow{x}\in R^n_{t(y)} \Leftrightarrow (g_1(\overrightarrow{x}), \cdots, g_n(\overrightarrow{x}))\in R^n_{y}.\]

Define $h_i(x,y)=$ $g_i(f_1(t(x),t(y)), \cdots, f_n(t(x),t(y)))$ for $1\leq i\leq n$.
\begin{claim}
$(A,B)$ is $\sf KP$ under $H(x,y)=(h_1(x,y), \cdots, h_n(x,y))$.
\end{claim}
\begin{proof}\label{}
Note that $H(x,y)$ is a $n$-ary recursive functional on $\mathbb{N}^2$, and
\begin{align*}
 H(x,y)\in R^n_{y}-R^n_{x} &\Rightarrow (f_1(t(x),t(y)), \cdots, f_n(t(x),t(y)))\in R^n_{t(y)}-R^n_{t(x)} \\
      &\Rightarrow F(t(x),t(y))\in C \\
      &  \Rightarrow H(x,y)\in A.
\end{align*}
Similarly, we can show that if $H(x,y)\in R^n_{x}-R^n_{y}$, then $H(x,y)\in B$.
Thus, $(A,B)$ is $\sf KP$.
\end{proof}
\end{proof}

Now we introduce the notions of \emph{$\sf CEI$} and \emph{$\sf DG$}.
\begin{definition}\label{}
Let $(A,B)$ be a disjoint pair  of $n$-ary RE relations.
\begin{enumerate}[(1)]
  \item  We say $(A,B)$ is \emph{$\sf CEI$} if there exists a recursive $n$-ary functional \[F(x,y)=(f_1(x,y), \cdots, f_n(x,y))\] on $\mathbb{N}^2$ such that for any $x,y\in\omega$, if $A\subseteq R^n_x$ and $B\subseteq R^n_y$, then
      \[F(x,y)\in R^n_x \Leftrightarrow F(x,y)\in R^n_y.\]
  \item We say $(A,B)$ is \emph{$\sf DG$} if there exists a recursive $n$-ary functional \[F(x,y)=(f_1(x,y), \cdots, f_n(x,y))\] on $\mathbb{N}^2$ such that for any $x,y\in\omega$, if $R^n_x\cap R^n_y=\emptyset$, then $F(x,y)\in A \Leftrightarrow F(x,y)\in R^n_x$ and $F(x,y)\in B \Leftrightarrow F(x,y)\in R^n_y$.
\end{enumerate}
\end{definition}

Now we show that $\sf KP$ implies $\sf CEI$.

\begin{proposition}\label{KP to CEI}
Let $(A,B)$ be a disjoint pair  of $n$-ary RE relations.  If $(A,B)$ is $\sf KP$, then $(A,B)$ is $\sf CEI$.
\end{proposition}
\begin{proof}\label{}
Suppose $(A,B)$ is $\sf KP$ under a recursive $n$-ary functional $F(x,y)$ on $\mathbb{N}^2$. We show $(A,B)$ is $\sf CEI$ under $F(x,y)$.

Suppose $A\subseteq R^n_x$ and $B\subseteq R^n_y$. Since $F(x,y)\in R^n_y-R^n_x \Rightarrow F(x,y)\in A\Rightarrow F(x,y)\in R^n_x$, we have $F(x,y)\notin R^n_y-R^n_x$. Since $F(x,y)\in R^n_x-R^n_y \Rightarrow F(x,y)\in B\Rightarrow F(x,y)\in R^n_y$, we have
$F(x,y)\notin R^n_x-R^n_y$. Thus, we have
\[F(x,y)\in R^n_x \Leftrightarrow F(x,y)\in R^n_y.\]
\end{proof}

Now we show that $\sf CEI$ implies $\sf DG$.

\begin{proposition}
Let $(A,B)$ be a disjoint pair  of $n$-ary RE relations.  If $(A,B)$ is $\sf CEI$, then $(A,B)$ is $\sf DG$.
\end{proposition}
\begin{proof}\label{}
It is easy to check that if $(A,B)$ is $\sf CEI$, then $(B,A)$ is also $\sf CEI$. Suppose $(B,A)$ is $\sf CEI$ under a recursive $n$-ary functional $F(x,y)=(f_1(x,y), \cdots, f_n(x,y))$ on $\mathbb{N}^2$, i.e., for any $x,y\in\omega$, if $B\subseteq R^n_x$ and $A\subseteq R^n_y$, then
\begin{equation}\label{key eq}
F(x,y)\in R^n_x \Leftrightarrow F(x,y)\in R^n_y.
\end{equation}

By s-m-n theorem, there exist recursive functions $t_1(y)$ and $t_2(y)$ such that for any $x$, we have $R^n_{t_1(x)}=R^n_{x}\cup A$ and $R^n_{t_2(x)}=R^n_{x}\cup B$.
Define $G(x,y)=(g_1(x,y), \cdots, g_n(x,y))$ where $g_i(x,y)=f_i(t_2(x),t_1(y))$ for $1\leq i\leq n$. Note that $G(x,y)$ is a recursive $n$-ary functional  on $\mathbb{N}^2$.
We show that $(A,B)$ is $\sf DG$ under $G(x,y)$.

Note that $G(x,y)=F(t_2(x), t_1(y))$.  For any $x, y\in\omega$, since $B\subseteq R^n_{t_2(x)}$ and $A\subseteq R^n_{t_1(y)}$, by (\ref{key eq}), we have:
\begin{equation}\label{eq1}
G(x,y)\in R^n_{t_2(x)}\Leftrightarrow G(x,y)\in R^n_{t_1(y)}.
\end{equation}

Assume $R^n_x\cap R^n_y=\emptyset$.
We show that \[G(x,y)\in A\Leftrightarrow G(x,y)\in R^n_{x}.\]
Suppose that $G(x,y)\in A$. Then $G(x,y)\in R^n_{t_1(y)}$. By (\ref{eq1}), $G(x,y)\in R^n_{t_2(x)}$. Then $G(x,y)\in R^n_{x}\cup B$. Thus, $G(x,y)\in R^n_{x}$.

Suppose $G(x,y)\in R^n_{x}$. Then $G(x,y)\in R^n_{t_2(x)}$. By (\ref{eq1}), $G(x,y)\in R^n_{t_1(y)}$. Then $G(x,y)\in  R^n_{y}\cup A$. Thus, $G(x,y)\in A$.

By a similar argument, we can show that $G(x,y)\in B\Leftrightarrow G(x,y)\in R^n_{y}$. So $(A,B)$ is $\sf DG$ under $G(x,y)$.
\end{proof}

Now we introduce the notion of $\sf DG$ relative to a collection of disjoint pairs  of $n$-ary RE relations. 

\begin{definition}\label{def of DG}
Let $\mathcal{C}$ be a collection of disjoint pairs  of $n$-ary RE relations, and  $(A,B)$  be  a disjoint pair  of $n$-ary RE relations.
\begin{enumerate}[(1)]
  \item We say $(A,B)$ is \emph{$\sf DG$ relative to $\mathcal{C}$} if there is a recursive $n$-ary functional $F(x,y)=(f_1(x,y), \cdots, f_n(x,y))$ on $\mathbb{N}^2$ such that for any $i,j\in\omega$, if $(R^n_i,R^n_j)\in \mathcal{C}$,  then $F(i,j)\in R^n_i\Leftrightarrow F(i,j)\in A$ and $F(i,j)\in R^n_j\Leftrightarrow F(i,j)\in B$.
  \item Let  $F(x,y)$ be a recursive $n$-ary functional on $\mathbb{N}^2$. We define conditions $C_1$-$C_3$ as follows:
      \begin{description}
        \item[$C_1$] for any $i,j\in\omega$, if $R^n_i=\mathbb{N}^n$ and $R^n_j=\emptyset$, then $F(i,j)\in A$;
        \item[$C_2$] for any $i,j\in\omega$, if $R^n_i=\emptyset$ and $R^n_j=\mathbb{N}^n$, then $F(i,j)\in B$;
        \item[$C_3$] for any $i,j\in\omega$, if $R^n_i=R^n_j=\emptyset$, then $F(i,j)\notin A\cup B$.
      \end{description}
      \item Define $\mathcal{D}=\{(\mathbb{N}^n, \emptyset), (\emptyset, \mathbb{N}^n),(\emptyset, \emptyset)\}$.
\end{enumerate}
\end{definition}

It is easy to check that $(A,B)$ is $\sf DG$ relative to $\mathcal{D}$  under $F(x,y)$ iff $C_1$-$C_3$ hold.

Now we show that $\sf DG$ relative to $\mathcal{D}$ implies $\sf DU$.

\begin{theorem}\label{}
Let $(A,B)$  be  a disjoint pair  of $n$-ary RE relations. If $(A,B)$ is $\sf DG$ relative to $\mathcal{D}$, then $(A,B)$ is $\sf DU$.
\end{theorem}
\begin{proof}\label{}
Suppose $(A,B)$ is $\sf DG$ relative to $\mathcal{D}$ under a recursive $n$-ary functional $F(x,y)=(f_1(x,y), \cdots, f_n(x,y))$ on $\mathbb{N}^2$. Then  $C_1$-$C_3$ hold. Let $(C,D)$  be  any disjoint pair  of $n$-ary RE relations. We show that $(C,D)$ is reducible to $(A,B)$.

\begin{claim}
For any $n$-ary RE relation $A$, there is a $n$-ary recursive function $t(\overrightarrow{y})$ such that for any $\overrightarrow{y}\in \mathbb{N}^n$,
\begin{enumerate}[(1)]
  \item  if $\overrightarrow{y}\in A$, then $R^n_{t(\overrightarrow{y})}=\mathbb{N}^n$;
  \item if $\overrightarrow{y}\notin A$, then $R^n_{t(\overrightarrow{y})}=\emptyset$.
\end{enumerate}
\end{claim}
\begin{proof}\label{}
Define $2n$-ary RE relation $M(\overrightarrow{x}, \overrightarrow{y})\triangleq \overrightarrow{y}\in A$. By s-m-n theorem, there exists a $n$-ary recursive function $t(\overrightarrow{y})$ such that $\overrightarrow{x}\in R^n_{t(\overrightarrow{y})} \Leftrightarrow M(\overrightarrow{x}, \overrightarrow{y})\Leftrightarrow \overrightarrow{y}\in A$.
Thus, if $\overrightarrow{y}\in A$, then $R^n_{t(\overrightarrow{y})}=\mathbb{N}^n$, and if $\overrightarrow{y}\notin A$, then $R^n_{t(\overrightarrow{y})}=\emptyset$.
\end{proof}
By the above claim, there are $n$-ary recursive functions $t_1(\overrightarrow{x})$  and $t_2(\overrightarrow{x})$ such that for any $\overrightarrow{x}\in \mathbb{N}^n$,
\begin{enumerate}[(1)]
  \item $\overrightarrow{x}\in C\Rightarrow R^n_{t_1(\overrightarrow{x})}=\mathbb{N}^n$;
  \item $\overrightarrow{x}\notin C\Rightarrow R^n_{t_1(\overrightarrow{x})}=\emptyset$;
  \item $\overrightarrow{x}\in D\Rightarrow R^n_{t_2(\overrightarrow{x})}=\mathbb{N}^n$;
  \item $\overrightarrow{x}\notin D\Rightarrow R^n_{t_2(\overrightarrow{x})}=\emptyset$.
\end{enumerate}

Define $G(\overrightarrow{x})=(g_1(\overrightarrow{x}), \cdots, g_n(\overrightarrow{x}))$ where $g_i(\overrightarrow{x})=f_i(t_1(\overrightarrow{x}), t_2(\overrightarrow{x}))$. Note that $G(\overrightarrow{x})= F(t_1(\overrightarrow{x}), t_2(\overrightarrow{x}))$ and $G(\overrightarrow{x})$ is a recursive $n$-ary functional on $\mathbb{N}^n$.

\begin{claim}
$G(\overrightarrow{x})$ is a reduction from $(C,D)$  to $(A,B)$.
\end{claim}
\begin{proof}\label{}
Suppose $\overrightarrow{x}\in C$. Then $R^n_{t_1(\overrightarrow{x})}=\mathbb{N}^n$ and $R^n_{t_2(\overrightarrow{x})}=\emptyset$. By the condition $C_1$ in Definition \ref{def of DG}, $F(t_1(\overrightarrow{x}), t_2(\overrightarrow{x}))\in A$. Thus, $G(\overrightarrow{x}) \in A$.

Suppose $\overrightarrow{x}\in D$. Then $R^n_{t_1(\overrightarrow{x})}=\emptyset$ and $R^n_{t_2(\overrightarrow{x})}=\mathbb{N}^n$. By the condition $C_2$ in Definition \ref{def of DG},  $F(t_1(\overrightarrow{x}), t_2(\overrightarrow{x}))\in B$. Thus, $G(\overrightarrow{x}) \in B$.

Suppose $\overrightarrow{x}\notin C\cup D$. Then $R^n_{t_1(\overrightarrow{x})}=R^n_{t_2(\overrightarrow{x})}=\emptyset$. By the condition  $C_3$ in Definition \ref{def of DG},  $G(\overrightarrow{x})= F(t_1(\overrightarrow{x}), t_2(\overrightarrow{x}))\notin A\cup B$. Thus, $(A,B)$ is $\sf DU$.
\end{proof}
\end{proof}

Since we have proven that semi-$\sf DU\Rightarrow \sf KP\Rightarrow \sf CEI\Rightarrow \sf DG\Rightarrow \sf DG$ relative to $\mathcal{D}\Rightarrow \sf DU$, thus semi-$\sf DU$ implies $\sf DU$.

\begin{corollary}\label{equ of meta property}
The following notions are equivalent:
\begin{enumerate}[(1)]
  \item Semi-$\sf DU$;
  \item $\sf KP$;
  \item $\sf CEI$;
  \item $\sf EI$;
  \item $\sf WEI$;
  \item $\sf DG$;
  \item $\sf DU$.
\end{enumerate}
\end{corollary}
\begin{proof}\label{}
For a disjoint pair of $n$-ary RE relations, we have proved that Semi-$\sf DU\Rightarrow \sf KP\Rightarrow \sf CEI\Rightarrow \sf DG\Rightarrow \sf DU$ and $\sf WEI\Rightarrow \sf DU$.
 Clearly, $\sf CEI\Rightarrow \sf EI\Rightarrow \sf WEI$ and $\sf DU\Rightarrow$ Semi-$\sf DU$.
Thus, the above notions are equivalent.
\end{proof}

This proof of Theorem \ref{semi-DU is DU} does not use any  version of recursion theorem. One merit of this proof is that it establishes that meta-mathematical properties in Corollary \ref{equ of meta property} are equivalent.

\section{The third proof of Theorem \ref{semi-DU is DU}}\label{3rd proof}

In this Appendix, we give a third proof of Theorem \ref{semi-DU is DU}. This proof  is simper than the second proof and does not use any version of recursion theorem. In this proof, we generalize the notion of separation functions  introduced in \cite{Smullyan93} to $n$-ary functionals on $\mathbb{N}^{n+2}$. Our proof is done in two steps: for disjoint pairs of $n$-ary RE relations, we first show that  semi-$\sf DU$ implies having a separation functional,  then we show that having a  separation functional implies $\sf DU$.
We first introduce the notion of separation functional. 

\begin{definition}\label{}
Let $(A,B)$  be  a disjoint pair  of $n$-ary RE relations. We say a $n$-ary functional $S(x, \overrightarrow{y}, z): \mathbb{N}^{n+2}\rightarrow \mathbb{N}^{n}$ on $\mathbb{N}^{n+2}$ is a \emph{separation functional} for $(A,B)$ if $S$ is recursive and for any $(n+1)$-ary RE relations $M_1(\overrightarrow{x},y)$ and $M_2(\overrightarrow{x},y)$, there is $h$ such that for any $y\in\omega$ and $\overrightarrow{x}\in \mathbb{N}^{n}$, we have:
\begin{enumerate}[(1)]
  \item $M_1(\overrightarrow{x},y)\wedge\neg M_2(\overrightarrow{x},y)\Rightarrow S(h, \overrightarrow{x}, y)\in A$;
  \item $M_2(\overrightarrow{x},y)\wedge\neg M_1(\overrightarrow{x},y)\Rightarrow S(h, \overrightarrow{x}, y)\in B$.
\end{enumerate}
\end{definition}

\begin{proposition}\label{there exist DU}
There is a pair $(U_1, U_2)$ of $n$-ary RE relations which is $\sf DU$.
\end{proposition}
\begin{proof}\label{}
Recall that for any $x\in \mathbb{N}$, $\overbrace{x}$ denotes $(x, \cdots, x)\in \mathbb{N}^{n}$. Define
\[U_1=\{\overbrace{J_{n+1}(J_2(x,y),\overrightarrow{z})}: \overrightarrow{z}\in R^n_y \, \text{before}\, \overrightarrow{z}\in R^n_x\}\]
and \[U_2=\{\overbrace{J_{n+1}(J_2(x,y),\overrightarrow{z})}: \overrightarrow{z}\in R^n_x \, \text{before}\, \overrightarrow{z}\in R^n_y\}.\]

Suppose $R^n_i\cap R^n_j=\emptyset$. Note that
\begin{align*}
 \overrightarrow{x}\in R^n_i & \Leftrightarrow B(\overrightarrow{x},i,j) \\
      &\Leftrightarrow \overrightarrow{x}\in R^n_i \, \text{before}\, \overrightarrow{x}\in R^n_j  \\
      &\Leftrightarrow \overbrace{J_{n+1}(J_2(j,i),\overrightarrow{x})}\in U_1;\\
      \overrightarrow{x}\in R^n_j & \Leftrightarrow B(\overrightarrow{x},j,i)\\
                             &\Leftrightarrow \overrightarrow{x}\in R^n_j \, \text{before}\,\overrightarrow{x}\in R^n_i\\
                             &\Leftrightarrow \overbrace{J_{n+1}(J_2(j,i),\overrightarrow{x})}\in U_2.
\end{align*}

Define $F(\overrightarrow{x})=\overbrace{J_{n+1}(J_2(j,i),\overrightarrow{x})}$. Then $F(\overrightarrow{x})$ is a reduction of $(R^n_i, R^n_j)$ to $(U_1, U_2)$.
\end{proof}

\begin{lemma}\label{lemma of SF}
For any $n$-ary RE relations $A$ and $B$, there is $h$ such that $F(\overrightarrow{x})\triangleq\overbrace{J_{n+1}(h,\overrightarrow{x})}$ is a semi-reduction from $(A-B, B-A)$ to $(U_1, U_2)$.
\end{lemma}
\begin{proof}\label{}
Suppose $A=R^n_i$ and $B=R^n_j$. Let $h=J_2(j,i)$. From the proof of Proposition \ref{there exist DU}, we have
\[\overrightarrow{x}\in A-B\Rightarrow B(\overrightarrow{x},i,j)\Rightarrow \overbrace{J_{n+1}(J_2(j,i),\overrightarrow{x})}\in U_1;\]
and
\[\overrightarrow{x}\in B-A\Rightarrow B(\overrightarrow{x},j,i)\Rightarrow \overbrace{J_{n+1}(J_2(j,i),\overrightarrow{x})}\in U_2.\]
Thus, $F(\overrightarrow{x})\triangleq\overbrace{J_{n+1}(h,\overrightarrow{x})}$ is a semi-reduction from $(A-B, B-A)$ to $(U_1, U_2)$.
\end{proof}

Now we show that semi-$\sf DU$ implies having a separation functional. 

\begin{theorem}\label{semi-DU is SF}
Let $(A,B)$  be  a disjoint pair  of $n$-ary RE relations. If $(A,B)$ is semi-$\sf DU$, then $(A,B)$ has a separation functional.
\end{theorem}
\begin{proof}\label{}
Let $(U_1, U_2)$ be the $\sf DU$ pair defined in Proposition \ref{there exist DU}. Suppose $(A,B)$ is semi-$\sf DU$. Then there is a recursive $n$-ary functional $G(\overrightarrow{x})$ on $\mathbb{N}^{n}$ such that $G(\overrightarrow{x})$ is a semi-reduction from $(U_1, U_2)$ to $(A,B)$.
Define $S(x, \overrightarrow{y}, z)=G(\overbrace{J_{n+1}(x,\overbrace{J_{n+1}(\overrightarrow{y},z)})})$. Note that $S(x, \overrightarrow{y}, z)$ is a recursive $n$-ary functional on $\mathbb{N}^{n+2}$.

\begin{claim}
$S(x, \overrightarrow{y}, z)$ is a separation functional for $(A,B)$.
\end{claim}
\begin{proof}\label{}
Take any $(n+1)$-ary RE relations $M_1(\overrightarrow{x},y)$ and $M_2(\overrightarrow{x},y)$. Define $C=\{\overbrace{J_{n+1}(\overrightarrow{x},y)}: M_1(\overrightarrow{x},y)\}$ and $D=\{\overbrace{J_{n+1}(\overrightarrow{x},y)}: M_2(\overrightarrow{x},y)\}$.

Apply Lemma \ref{lemma of SF} to $C$ and $D$. Then there is $h$ such that $F(\overrightarrow{x})=\overbrace{J_{n+1}(h,\overrightarrow{x})}$ is a semi-reduction from $(C-D, D-C)$ to $(U_1, U_2)$. Thus, $G(F(\overrightarrow{x}))$ is a semi-reduction from $(C-D, D-C)$ to $(A,B)$. Then for any $y\in\omega$ and  $\overrightarrow{x}\in \mathbb{N}^n$, we have:
\begin{enumerate}[(1)]
  \item $M_1(\overrightarrow{x},y)\wedge\neg M_2(\overrightarrow{x},y)\Rightarrow \overbrace{J_{n+1}(\overrightarrow{x},y)}\in C-D\Rightarrow G(F(\overbrace{J_{n+1}(\overrightarrow{x},y)})) \in A\Rightarrow  G(\overbrace{J_{n+1}(h,\overbrace{J_{n+1}(\overrightarrow{x},y)})})\in A
      \Rightarrow S(h, \overrightarrow{x}, y)\in A$;
  \item $M_2(\overrightarrow{x},y)\wedge\neg M_1(\overrightarrow{x},y)\Rightarrow \overbrace{J_{n+1}(\overrightarrow{x},y)}\in D-C\Rightarrow G(F(\overbrace{J_{n+1}(\overrightarrow{x},y)})) \in B\Rightarrow S(h, \overrightarrow{x}, y)\in B$.
\end{enumerate}
\end{proof}
\end{proof}

Now we show that having a  separation functional implies $\sf DU$. 
\begin{theorem}\label{SF is DU}
Let $(A,B)$  be  a disjoint pair  of $n$-ary RE relations. If $(A,B)$ has a separation functional, then $(A,B)$ is $\sf DU$.
\end{theorem}
\begin{proof}\label{}
Suppose $S(x, \overrightarrow{y}, z)$ is a $n$-ary separation functional on $\mathbb{N}^{n+2}$ for $(A,B)$. Let $(C,D)$ be a disjoint pair  of $n$-ary RE relations. We show that $(C,D)$ is reducible to $(A,B)$. Define  $M_1(\overrightarrow{x},y)\triangleq  \overrightarrow{x} \in C \vee S(y, \overrightarrow{x}, y)\in B$, and $M_2(\overrightarrow{x},y)\triangleq \overrightarrow{x} \in D \vee S(y, \overrightarrow{x}, y)\in A$. Then there is $h$ such that for any $y\in\omega$  and  $\overrightarrow{x}\in \mathbb{N}^{n}$, we have:  \[M_1(\overrightarrow{x},y)\wedge \neg M_2(\overrightarrow{x},y)\Rightarrow S(h, \overrightarrow{x}, y)\in A\]
and \[M_2(\overrightarrow{x},y)\wedge \neg M_1(\overrightarrow{x},y)\Rightarrow S(h, \overrightarrow{x}, y)\in B.\]

Let $y\triangleq h$. Then
\begin{enumerate}[(1)]
  \item $[\overrightarrow{x} \in C \vee S(h, \overrightarrow{x}, h)\in B]\wedge \neg [\overrightarrow{x} \in D \vee S(h, \overrightarrow{x}, h)\in A]\Rightarrow S(h, \overrightarrow{x}, h)\in A$;
  \item $[\overrightarrow{x} \in D \vee S(h, \overrightarrow{x}, h)\in A]\wedge \neg [\overrightarrow{x} \in C \vee S(h, \overrightarrow{x}, h)\in B]\Rightarrow S(h, \overrightarrow{x}, h)\in B$.
\end{enumerate}
Thus, from (1)-(2), we have:
\begin{align*}
\overrightarrow{x} \in C\Leftrightarrow   S(h, \overrightarrow{x}, h)\in A;\\
\overrightarrow{x} \in D\Leftrightarrow   S(h, \overrightarrow{x}, h)\in B.
\end{align*}
Define $F(\overrightarrow{x})=S(h, \overrightarrow{x}, h)$. Note that $F(\overrightarrow{x})$ is a $n$-ary recursive functional on $\mathbb{N}^n$ and $F(\overrightarrow{x})$ is a reduction function from $(C,D)$ to $(A,B)$.
\end{proof}

As a corollary of Theorem \ref{semi-DU is SF} and Theorem \ref{SF is DU}, we have semi-$\sf DU$ implies $\sf DU$.

\end{document}